%% file: rev7.tex
\documentclass[a4paper]{amsart}
\usepackage{epsfig}
\usepackage{amssymb}
\usepackage{amsmath}
\usepackage{latexsym}

\usepackage{latexsym,graphicx, color}
\newtheorem{thm}{Theorem}[section] \newtheorem{lem}[thm]{Lemma}
\newtheorem{cor}[thm]{Corollary} \newtheorem{conj}[thm]{Conjecture}
\newtheorem{prop}[thm]{Proposition} \theoremstyle{definition}
\newtheorem{defn}[thm]{Definition}

\newtheorem{conv}[thm]{Convention} \newtheorem{rem}[thm]{Remark}

\newtheorem{cond}[thm]{Condition}

\def\strutdepth{\dp\strutbox}
\def \ss{\strut\vadjust{\kern-\strutdepth \sss}}
\def \sss{\vtop to \strutdepth{
\baselineskip\strutdepth\vss\llap{$\diamondsuit\;\;$}\null}}

\def\strutdepth{\dp\strutbox}
\def \sst{\strut\vadjust{\kern-\strutdepth \ssss}}
\def \ssss{\vtop to \strutdepth{
\baselineskip\strutdepth\vss\llap{$\spadesuit\;\;$}\null}}

\begin{document}

\date{}

\author[I.~Kapovich]{Ilya Kapovich}

\address{\tt Department of Mathematics, University of Illinois at
  Urbana-Champaign, 1409 West Green Street, Urbana, IL 61801, USA
  \newline http://www.math.uiuc.edu/\~{}kapovich/} \email{\tt
  kapovich@math.uiuc.edu}

\author[R.~Weidmann]{Richard Weidmann}

\address{\tt Mathematisches Seminar, Christian-Albrechts-Universit\"at zu Kiel, Ludewig-Meyn Str. 4, 24098 Kiel, Germany}
\email{\tt weidmann@math.uni-kiel.de}

\title{Nielsen equivalence in small cancellation groups}

\begin{abstract} Let $G$ be a group given by the presentation
  \[\langle a_1,\ldots,a_k,b_1,\ldots b_k\,|\,a_i=u_i(\bar b),
  b_i=v_i(\bar a)\hbox{ for }1\le i\le k\rangle,\] where $k\ge 2$ and
  where the $u_i\in F(b_1,\dots, b_k)$ and
  $w_i\in F(a_1,\dots, a_k)$ are random words. Generically such a group is a small
  cancellation group and it is clear  that $(a_1,\ldots,a_k)$ and
  $(b_1,\ldots,b_k)$ are generating $n$-tuples for $G$. We prove that for
  generic choices of $u_1,\dots, u_k$ and $v_1,\dots, v_k$ the
  ``once-stabilized'' tuples $(a_1,\ldots, a_k,1)$ and
  $(b_1,\ldots,b_k,1)$ are not Nielsen equivalent in $G$. This
  provides a counter-example for a Wiegold-type conjecture in the
  setting of word-hyperbolic groups. We conjecture that in the above
  construction at least $k$ stabilizations are needed to make the
  tuples  $(a_1,\ldots, a_k)$ and $(b_1,\dots,b_k)$ Nielsen equivalent.
\end{abstract}

\thanks{The first author was supported by the NSF
  grant DMS-0904200}
\subjclass[2000]{Primary 20F65, Secondary 20F67, 57M07}
\keywords{Nielsen equivalence, word-hyperbolic groups, genericity}  

\maketitle

\section{Introduction}

For a group $G$ and an integer $k\ge 1$, Nielsen equivalence is an
equivalence relation on the set of ordered $k$-tuples of elements of
$G$. Let $\mathcal T=(g_1,\ldots,g_k)\in G^k$ and $\mathcal
T'=(g'_1,\ldots,g'_k)\in G^k$ be two $k$-tuples. Then $\mathcal T$ and $\mathcal T'$ are {\em elementary equivalent} (write $\mathcal T\sim_e\mathcal T'$) if one of the following holds:

\begin{enumerate}
\item There exists some $\sigma\in S_k$ such that $g_i'=g_{\sigma(i)}$
  for $1\le i\le k$.
\item $g_i'=g_i^{-1}$ for some $i=1,\ldots,k$ and $g_j'=g_j$ for
  $j\neq i$.
\item $g_i'=g_ig_j$ for some $i\neq j$ and $g_t'=g_t$ for $t\neq i$.
\end{enumerate}

The above transformations are called a {\em Nielsen transformations}
or {\em Nielsen moves}. Two tuples $\mathcal T$ and $\mathcal T'$ are
called {\em Nielsen equivalent} or simply {\em equivalent} if there
exists some finite sequence $\mathcal T_0,\mathcal T_1,\ldots,
\mathcal T_l$ such that 
\[
\mathcal T=\mathcal T_0\sim_e\mathcal T_1\sim_e\ldots\sim_e\mathcal
T_l=\mathcal T'.
\]

Jacob Nielsen introduced this equivalence relation to study subgroups of
free groups. Among other things, he showed that in the free group
$F_n=F(x_1,\ldots,x_n)$  any generating $k$-tuple is Nielsen
equivalent to the tuple $(x_1,\ldots,x_n,1,\ldots,1)$ (so that,
necessarily $k\ge n$). In particular any two generating $k$-tuples of $F_n$ are
Nielsen equivalent. This result, together with the fact that any
Nielsen move on a basis of $F_n$ induces an automorphism of $F_n$, implies the following alternative characterization of Nielsen equivalence:

\medskip Two $k$-tuples $\mathcal T=(g_1,\ldots,g_k)\in G^k$ and $\mathcal T'=(g'_1,\ldots,g'_k)\in G^k$ are Nielsen equivalent if and only if there exists a homomorphism $\phi:F_k\to G$ and an automorphism $\alpha$ of $F_k$ such that the following hold:

\begin{enumerate}
\item $g_i=\phi(x_i)$ for $1\le i\le k$.
\item $g_i'=\phi\circ\alpha(x_i)$ for $1\le i\le k$.
\end{enumerate}

Let $F_k=F(x_1,\dots, x_k)$ be a free group of rank $k$ with a fixed
free basis $(x_1,\dots, x_k)$. There is a natural identification
between the set $Hom(F_k,G)$ of homomorphisms from $F_k$ to $G$ and
the set $G^k$ of $k$-tuples of elements of $G$. There is also a
natural left action of $Aut(F_k)$ on $Hom(F_k,G)$ by
pre-composition. In view of the above remark, two $k$-tuples of
elements of $G$ are Nielsen equivalent if and only if the
corresponding elements of $Hom(F_k,G)$ lie in the same
$Aut(F_k)$-orbit.

In general it is difficult to decide if two $k$-tuples are
Nielsen equivalent in a given group.

If two tuples $\mathcal T=(g_1,\ldots,g_k)\in G^k$ and $\mathcal
T'=(g'_1,\ldots,g'_k)\in G^k$ are Nielsen equivalent then they
generate the same subgroup of $G$, that is $\langle \mathcal
T\rangle=\langle \mathcal T'\rangle\le G$. Thus if two tuples generate
different subgroups of $G$, the tuples are not Nielsen equivalent.
However, this observation does not help in distinguishing
Nielsen equivalence classes of tuples generating the same subgroup, in
particular those tuples that generate the entire group $G$ (\emph{generating}
tuples). Under the identification of the set of $k$-tuples in $G$ with
$Hom(F_k,G)$ discussed above, the set of generating $k$-tuples of $G$
corresponds to the set $Epi(F_k,G)$ of epimorphisms from $F_k$ to $G$.

The only exception is the case $k=2$. A basic fact due to Nielsen
shows that if 2-tuples $(g_1,g_2)$ and $(h_1,h_2)$ are
Nielsen equivalent in $G$ then $[g_1,g_2]$ is conjugate to
$[h_1,h_2]^{\pm 1}$ in $G$. No such criteria exist for $k\ge 3$ and
there are few known results distinguishing Nielsen-equivalence
classes of generating $k$-tuples for $k\ge 3$.

Note that even in the algorithmically nice setting of torsion-free
word-hyperbolic groups the problem of deciding if two tuples are
Nielsen equivalent is algorithmically undecidable.
Indeed the subgroup membership problem is a special case of this problem since
two tuples $(g_1,\ldots,g_n,h)$ and $(g_1,\ldots,g_n,1)$ are Nielsen
equivalent if and only if $h\in\langle g_1,\ldots, g_n\rangle$. This implies in
particular that Nielsen equivalence  is not decidable for finitely
presented torsion-free small cancellation groups since they do not have
decidable subgroup membership problem as shown by Rips  \cite{R}.

Understanding Nielsen equivalence of generating $k$-tuples for $k\ge 3$
is particularly difficult, and the problem becomes even harder if
$k>rank(G)$, where $rank(G)$ is the smallest size of a generating set
of $G$.

Of particular interest here is the so-called \emph{Wiegold conjecture} about
generating tuples of finite simple groups. A generating
$k$-tuple for $G$ is \emph{redundant} if it contains a proper subtuple that
still generates $G$. We say that a generating tuple is \emph{weakly
  redundant} if it is Nielsen equivalent to a redundant tuple. Note
that a redundant tuple $(g_1,\dots, g_k)$ is always Nielsen equivalent
to a $k$-tuple of the form $(h_1,\dots, h_{k-1},1)$. Thus a generating
tuple is weakly redundant if and only if it is Nielsen equivalent to a
tuple containing a trivial entry. It is well-known, as a consequence
of the classification of finite simple groups, that every finite
simple group $G$ is two-generated, so
that $rank(G)\le 2$. The Wiegold conjecture says that if $G$ is a
finite simple group and $k\ge 3$ then any two generating $k$-tuples of
$G$ are Nielsen equivalent; in other words the conjecture says that the action of $Aut(F_k)$ on
$Epi(F_k,G)$ is transitive in this case. Since $G$ is two-generated
and has a generating $k$-tuple of the form $(a,b,1,\dots, 1)$, this
implies that any generating $k$-tuple of $G$ is redundant. The Wiegold
conjecture is closely related to the so-called ``product replacement
algorithm'' for finite groups and there is substantial experimental
evidence and a number of partial theoretical results in favor of the
validity of the Wiegold conjecture. We refer the reader to ~\cite{Pak,LGM,LuPa}
for a more extensive discussion of this topic.

One obvious way of producing redundant tuples is via the so-called
``stabilization'' moves. A stabilization move on a $k$-tuple
$(g_1,\dots, g_k)$ gives a $(k+1)$-tuple $(g_1,\dots, g_k,1)$. It is
easy to see that for any generating $k$-tuples $(g_1,\dots, g_k)$ and
$(h_1,\dots, h_k)$ of a group $G$, applying $k$ stabilization moves to
each of them produces Nielsen equivalent $2k$-tuples $(g_1,\dots,
g_k,1,\dots, 1)$ and $(h_1,\dots, h_k,1,\dots, 1)$. The Wiegold conjecture
implies that for any two generating pairs $(a,b)$ and $(a_1,b_1)$ of a
finite simple group $G$, the once-stabilized tuples $(a,b,1)$ and
$(a_1,b_1,1)$ are Nielsen equivalent.

\medskip Let us mention here the (few) known results on distinguishing
Nielsen equivalence of generating $k$-tuples for infinite groups.
Apart from the special and much easier case of $k=2$, these can
be mostly separated into two distinct approaches.

\medskip The first one is algebraic ($K$-theoretic). The earliest work is due to Noskov \cite{No} who showed that there exist non-minimal generating tuples that are not Nielsen equivalent to a tuple containing the trivial element and thereby giving a negative answer to a question of Waldhausen. These results where later generalized by Evans \cite{E1}.

\smallskip Lustig and Moriah \cite{LM1}, \cite{LM2}, \cite{LM3} used algebraic methods to distinguish Nielsen equivalence classes of Fuchsian groups and other groups with appropriate presentations. This enabled them to distinguish isotopy classes of vertical Heegaard splittings of Seifert manifolds.

\smallskip Recently Evans \cite{E2,E3}, for any given number $N$, produced large generating tuples of metabelian groups that do not become Nielsen equivalent after adding the trivial element to the tuples $N$ times, making those the first examples of this type in the literature even for the case $N=1$. The generating tuples however are much bigger than the rank of the group.

\medskip The second approach is combinatorial-geometric and is closer in spirit to Nielsen's original work, it relies mostly on  using cancellation methods. First in line is Grushko's theorem \cite{G} which states that any generating tuple of a free product is Nielsen equivalent to a tuple of elements that lie in the union of the factors. Together with recent work of the second author \cite{W} this implies that Nielsen equivalence of irreducible generating tuples in a free product is decidable iff it is decidable in the factors.

\smallskip Zieschang \cite{Z} proves that any minimal generating tuple of a surface group is Nielsen equivalent to the standard generating tuple and proves a similar result for Fuchsian groups that lead to the solution of the rank problem \cite{PRZ}. Nielsen equivalence in Fuchsian groups has been studied by many authors. Recently Louder \cite{L} has generalized Zieschangs result to arbitrary generating tuples.

\smallskip The finiteness of Nielsen equivalence classes of $k$-tuples  for torsion-free locally quasiconvex-hyperbolic groups has been established by the authors \cite{KW} generalizing a result of Delzant \cite{D} who studied 2-generated groups. The first author and Schupp \cite{KS} have recently established uniqueness of the Nielsen equivalence class of minimal generating tuples for a class of groups closely related to the one studied in the present article.

\medskip The main result of this paper is the following theorem, which
implies in particular that there exist 2-generated torsion-free
word-hyperbolic groups that have generating pairs $(a_1,a_2)$ and $(b_1,b_2)$ such that $(a_1,a_2,1)$ and $(b_1,b_2,1)$ are not Nielsen equivalent. 
See Section~\ref{sect:gen} for precise definitions and notations related to
genericity. If $k\ge 2, m\ge 2$ are integers, for the free group
$F(a_1,\dots, a_k)$ we denote by $\mathcal C_{m,A}$ the set of all
$m$-tuples $(v_1,\dots, v_m)$ of cyclically reduced words in
$F(a_1,\dots, a_k)$ such that $|v_1|=\dots =|v_m|$. Similarly, denote
by  $\mathcal C_{k,B}$ the set of all $k$-tuples  $(u_1,\dots, u_k)$ of cyclically reduced words in
$F(b_1,\dots, b_k)$ such that $|u_1|=\dots =|u_k|$.

\begin{thm}\label{main} Let $k\ge 2, m\ge 2$ be arbitrary and let $G$ be a group given by the presentation
\[\langle a_1,\ldots,a_k,b_1,\ldots b_m\,|\,a_i=u_i(\bar b),
b_j=v_j(\bar a)\hbox{ for }1\le i\le k, 1\le j\le m \rangle. \tag{$\dag$}\]

There exist generic subsets $\, \mathcal U_k\subseteq \mathcal
C_{k,B}$ and $\mathcal V_m\subseteq \mathcal C_{m,A}$ such
that for any $(u_1,\dots, u_k)\in \mathcal U_k$ and any $(v_1,\dots,
v_m)\in \mathcal V_m$ with $|u_1|=|v_1|$, the group $G$ given by
presentation $(\dag)$ above has the following property:

For any $g_1,\dots, g_{k-1} \in G$ the generating $(k+1)$-tuple
$(a_1,\ldots, a_k,1)$ of $G$ is not Nielsen equivalent to a tuple
$(b_1,b_2,g_1,\ldots,g_{k-1})$.
\end{thm}

The theorem immediately implies that for the case $k=m$ the generating
$(k+1)$-tuples $(a_1,\ldots,a_k,1)$ and $(b_1,\ldots,b_k,1)$ are not
Nielsen equivalent in $G$. However, we do believe that much more is true, namely that following holds:

\begin{conj}\label{mainc} Let $k\ge m\ge 2$. There exist generic
  subsets $\, \mathcal U_k\subseteq \mathcal
C_{k,B}$ and $\mathcal V_m\subseteq \mathcal C_{m,A}$ such
that for any $(u_1,\dots, u_k)\in \mathcal{\widetilde U}_k$ and any $(v_1,\dots,
v_m)\in \mathcal{\widetilde V}_m$ with $|u_1|=|v_1|$, the group $G$ given by presentation $(\dag)$ above
has the following property:

For any $1\le t<m$ the generating $(k+t)$-tuple
$(a_1,\ldots, a_k,1,\ldots,1)$ is not equivalent to a $(k+t)$-tuple of type
$(b_1,\ldots,b_{{t+1}},g_1,\ldots,g_{k-1})$.
\end{conj}

In particular, if $k=m$, the conjecture would imply that at least $k$
stabilizations are needed in $G$ in order to make the generating
$k$-tuples $(a_1,\dots, a_k)$ and $(b_1,\dots, b_k)$ Nielsen
equivalent.

The conclusion of Theorem~\ref{main} implies that for $G$ as in the
assumption of the theorem with generic $u_i$ and $v_j$, the generating
$k$-tuples $(a_1,\dots, a_k)$ and\\ $(b_1,b_2,\dots, b_m,1,\dots,1)$
(where $(k-m)\ge 0$ trivial entries are present in the second tuple)  are not
Nielsen equivalent in $G$. While the proof of Theorem~\ref{main} is very
complicated, we also give a simple proof (see Theorem~\ref{specialcase}) that in this situation
$(a_1,\dots, a_k)$ is not Nielsen equivalent to a $k$-tuple of the
form $(b_1,,*,\ldots,*)$.

Although we do not prove it in this paper,
for the case $k=m\ge 2$ and $G$ as in Theorem~\ref{main} one can use
the methods of \cite{KS} to show that $G$ has \emph{exactly} two
Nielsen-equivalence classes of generating $k$-tuples, namely
$(a_1,\dots, a_k)$ and $(b_1,\dots, b_k)$.

We are grateful to the referee for useful comments.

\section{Small cancellation theory}

Recall that a set $R$ of cyclically reduced words in $F=F(a_1,\dots,
a_k)$ is \emph{symmetrized} if for every $r\in R$ all cyclic
permutations of $r^{\pm 1}$ belong to $R$. For a symmetrized set
$R\subseteq F(a_1,\dots, a_k)$, a freely reduced word $v\in
F(a_1,\dots, a_k)$ is a \emph{piece with respect to $R$} if there
exist $r_1,r_2\in R$, such that $r_1\ne r_2$ and such that $v$ is an
initial segment of each of $r_1,r_2$.

\begin{defn}[Small Cancellation Condition]
Let $R\subseteq F(a_1,\dots, a_k)$ be a symmetrized set of cyclically
reduced words. Let $0<\lambda<1$. We say that $R$ satisfies the
\emph{$C'(\lambda)$-small cancellation condition} or that $R$ is a
\emph{$C'(\lambda)$-set} if, whenever $v$
is a piece with respect to $R$ and $v$ is a subword of some $r\in
R$, then $|v|<\lambda |r|$.

We say that a presentation $G=\langle a_1,\dots, a_k | R\rangle$
satisfies the \emph{$C'(\lambda)$-small cancellation condition}  if
$R\subseteq F(a_1,\dots, a_k)$ is a $C'(\lambda)$ set.
\end{defn}

The following fact is a well-known basic property of small cancellation groups~\cite{LS}:

\begin{prop}\label{prop:green}
Let $G=\langle a_1,\dots, a_k | R\rangle$ be a
$C'(\lambda)$-presentation, where $\lambda\le 1/6$. Let $w\in
F(a_1,\dots, a_k)$ be a nontrivial freely reduced word such that
$w=_G 1$. Then $w$ has a subword $u$ such that for some $r\in R$ the
word $u$ is a subword of $r$ satisfying $|u|>(1-3\lambda)|r|$.
\end{prop}

\begin{defn}\label{defn:l-reduced}
Let $G=\langle a_1,\dots, a_k | R\rangle$ be a
$C'(\lambda)$-presentation, where $\lambda\le 1/100$. For a freely
reduced word $w\in F(a_1,\dots, a_k)$ we say that $w$ is \emph{Dehn
reduced with respect to $R$} if $w$ does not contain a subword $u$
such that $u$ is also a subword of some $r\in R$ with $|u|>|r|/2$.

We say that a freely reduced word $w\in F(a_1,\dots, a_k)$ is
\emph{$\lambda$-reduced with respect to $R$} if $w$ does not contain
a subword $u$ such that $u$ is also a subword of some $r\in R$ with
$|u|>(1-3\lambda)|r|$.

Similarly, we say that a cyclically reduced word $w\in F(a_1,\dots, a_k)$ is
\emph{$\lambda$-cyclically reduced with respect to $R$} if every
cyclic permutation of $w$ is $\lambda$-reduced with respect to $R$.
We also say that a cyclically reduced word $w\in F(a_1,\dots, a_k)$ is
\emph{cyclically Dehn-reduced with respect to $R$} if every
cyclic permutation of $w$ is Dehn-reduced with respect to $R$.

\end{defn}

Note that if $\lambda\le 1/6$ then any Dehn-reduced word is
$\lambda$-reduced. Proposition~\ref{prop:green} says that for a
$C'(\lambda)$-presentation with $\lambda\le 1/6$, any nontrivial freely
reduced and $\lambda$-reduced word in $F(a_1,\dots, a_k)$ represents a
nontrivial element of $G$.

The following statement follows from the basic results of small cancellation
theory, established in Ch. V, Sections 3-5 of \cite{LS} (see also \cite{Stre}).
\begin{prop}\label{prop:eq-conj}[Equality and Conjugacy diagrams in
  $C'(\lambda)$-groups]

Let \[ G=\langle a_1,\dots, a_k |
  R\rangle\tag{$\ast$} \] be a $C'(\lambda)$-presentation, where
  $\lambda\le 1/100$.

\begin{enumerate}

\item Let $w_1,w_2\in F(a_1,\dots, a_k)$ be freely reduced and
  $\lambda$-reduced words such that $w_1=_G w_2$. Then any reduced
  equality diagram $D$ over $(\ast)$, realizing the equality $w_1=_G
  w_2$, has the form as shown in Figure~\ref{Fi:eq}. Specifically, any
  region $Q$ of $D$ labeled by $r\in R$ intersects both the upper
  boundary of $D$ (labeled by $w_1$) and the lower boundary of $D$
  (labeled by $w_2$) in simple segments $\alpha_1, \alpha_2$
  accordingly, satisfying \[ \lambda|r|\le |\alpha_j|\le
  (1-3\lambda)|r|.  \] Moreover, if two regions $Q,Q'$ in $D$,
  labeled by $r,r'\in R$, have a common edge, then they intersect in
  closed simple segment $\gamma$ joining a point of the upper boundary
  of $D$ with a point of the lower boundary of $D$ and labeled by a
  piece with respect to $R$.  In particular $|\gamma|<\lambda |r|$ and
  $|\gamma|< \lambda|r'|$. Further, if both $w_1$ and $w_2$ are also
  Dehn-reduced, then for the segments $\alpha_j$ above we have
  $|\alpha_j|\ge (1/2-2\lambda)|r|\ge |r|/3$.

\begin{figure}[htb]
		\scalebox{.9}{\includegraphics{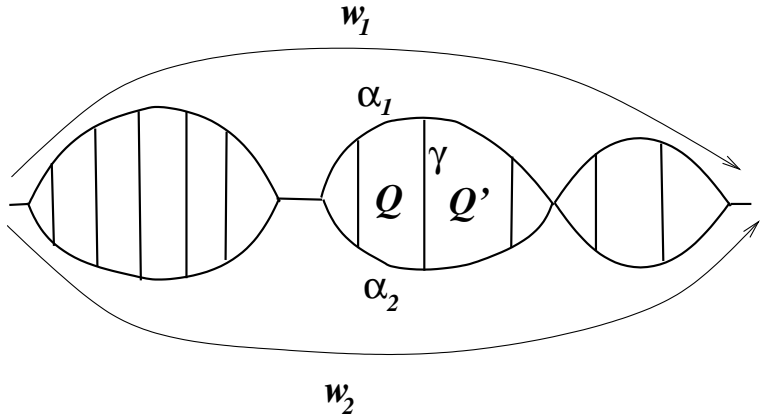}}
		\caption{Equality diagram in a small cancellation
                  group}
\label{Fi:eq}

\end{figure}

\item Let $w_1,w_2\in F(a_1,\dots, a_k)$ be nontrivial cyclically
    reduced and $\lambda$-cyclically reduced words
    representing conjugate elements of $G$. Then there exists a
    reduced conjugacy diagram $D$ over $(\ast)$ with the inner cycle
    boundary labeled by a cyclic permutation of $w_2$ and the outer
    cycle boundary labeled by a cyclic permutation of $w_1$, of the
    form shown in Figure~\ref{Fi:conj}.  Specifically, any region $Q$ of $D$
    labeled by $r\in R$ intersects both the outer boundary of $D$
    (labeled by a cyclic permutation of $w_1$) and the inner boundary
    of $D$ (labeled by a cyclic permutation of $w_2$) in simple
    segments $\alpha_1, \alpha_2$ accordingly, satisfying \[
    \lambda|r|\le |\alpha_j|\le (1-3\lambda)|r|.  \] Moreover, if two
    regions $Q,Q'$ in $D$, labeled by $r,r'\in R$, have a common
    edge, then they intersect in closed simple segment $\gamma$
    joining a point of the inner boundary of $D$ with a point of the
    outer boundary of $D$ and labeled by a piece with respect to $R$.
    In particular $|\gamma|<\lambda |r|$ and $|\gamma|<
    \lambda|r'|$. Further, if both $w_1$ and $w_2$ are also
    Dehn-reduced, then for the segments $\alpha_j$ above we have
    $|\alpha_j|\ge (1/2-2\lambda)|r|\ge |r|/3$.

\begin{figure}[htb]
		\scalebox{.7}{\includegraphics{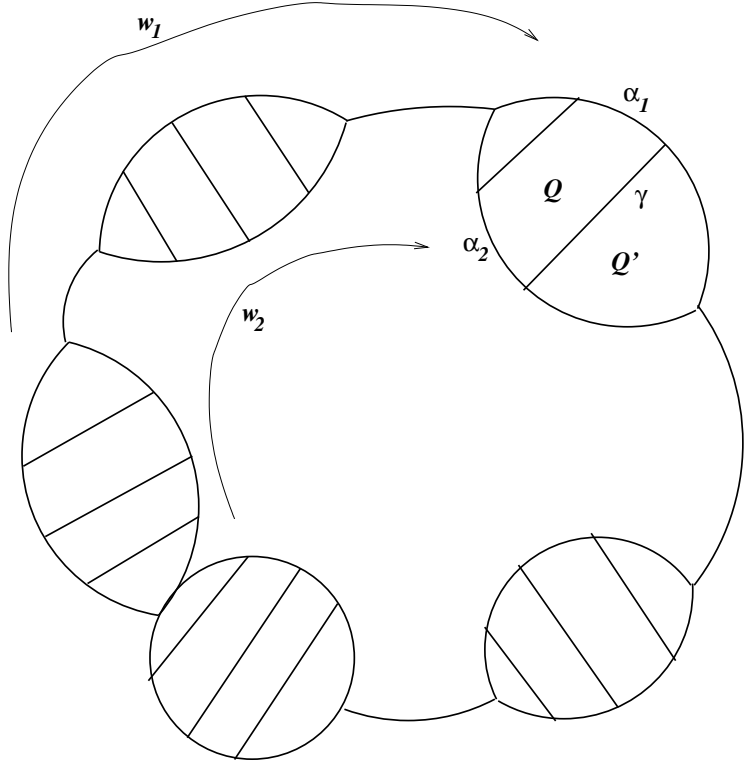}}
		\caption{Conjugacy diagram in a small cancellation
                  group}
\label{Fi:conj}

\end{figure}
 \end{enumerate}
\end{prop}

\begin{rem} Suppose that $w_1$ and $w_2$
  are as in part (1) of Proposition~\ref{prop:eq-conj} and suppose
  that $w_1$ and $w_2$ are Dehn-reduced (which is a stronger
  assumption than being $\lambda$-reduced). Let $Q$ be a region of $D$
  intersecting the upper and the lower boundaries of $D$ in segments
  $\alpha_1$ and $\alpha_2$. Let $r\in R$ be the label of a boundary
  cycle of $Q$. Since
  the overlaps of $Q$ with the neighboring regions in $D$ have length
  at most $\lambda|r|$ each and since $|\alpha_j|\le |r|/2$ by
  Dehn-reduceness of $w_1$ and $w_2$, it follows that \[ |\alpha_j|\ge
  |r|-2\lambda|r|-|r|/2=|r|(1/2-2\lambda)\ge |r|\frac{48}{100}\ge
  |r|/3, \] since $\lambda\le 1/100$. Thus in this case \[ |r|/3\le
  \alpha_j\le |r|/2.  \] Similar conclusions apply to case (2) of
  Proposition~\ref{prop:eq-conj} if we assume $w_1$ and $w_2$ to be
  cyclically Dehn reduced there.
\end{rem}

\begin{cor}\label{cor:conj-sc} Let $G=\langle
  a_1,\dots, a_k |R\rangle$ be a $C'(\lambda)$ presentation where
  $\lambda\le 1/100$.  Suppose that for every
    $r\in R$ we have $|r|\ge 2/\lambda+1$. Then:
\begin{enumerate}
\item For $i\ne j$ the
    elements $a_i^{\pm 1}$ and $a_j^{\pm 1}$ are not conjugate in $G$,
    and $a_i$ is not conjugate to $a_i^{-1}$ in $G$.
\item Suppose that wherever $a_s^p$ is a subword of some $r\in R$,
    where $1\le s\le k$, then $|p|< \lambda |r|$. Then for any $i\ne
    j$ and any $m\ne 0, n\ne 0$, the elements $a_i^m$ and $a_j^n$ are
    not conjugate in $G$.
\item Suppose that $w$ is a cyclically reduced word in $F(A)$ that is
  conjugate in $G$ to $a_1$. Then either $w=a_1$ in $F(A)$ or $w$ is not
  cyclically $\lambda$-reduced.
\end{enumerate}
\end{cor}

\begin{lem}\label{lem:sc-conj}
Let $G=\langle a_1,\dots, a_k |R\rangle$ be a $C'(\lambda)$ presentation where
$\lambda\le 1/100$. Let $d_G$ be the word-metric on $G$ corresponding
to the generating set $\{a_1,\dots, a_k\}$. Let $g\in G, g\ne 1$ be arbitrary. Then there
exist freely reduced words $u,v\in F(a_1,\dots, a_k)$ such that:
\begin{enumerate}
\item We have $g=_G v u v^{-1}$.
\item The word $u$ is cyclically reduced and cyclically Dehn-reduced. Moreover, $u$ is $d_G$-geodesic and the element of $G$ represented by $u$ is of shortest possible length among all elements conjugate to $g$ in $G$.
\item The word $v$ is Dehn-reduced.
\item The word $vuv^{-1}$ is freely reduced, as written, and is $\lambda$-reduced.
\item If $z$ is any $\lambda$-reduced word in $F(a_1,\dots, a_k)$
  representing $g$, then either $z=vuv^{-1}$ in $F(a_1,\dots, a_k)$ or
  there is some $r\in R$ such that both $z$ and $vuv^{-1}$ contain
  subwords of $r$ of length $\ge \lambda |r|$.
\end{enumerate}
\end{lem}
\begin{proof}

Consider all representations of $g$ as $g=hg_0h^{-1}$ where $g_0\in G$
is shortest in the conjugacy class of $g$. Once $g_0$ is fixed, among all such
representations of $g$ as $g=hg_0h^{-1}$, choose one, $g=hg_0h^{-1}$, where $h\in G$ is
the shortest possible.  Let $u$ be a $G$-geodesic representative of
$g_0$ and let $v$ be a $G$-geodesic representative of $h$.
Note that the minimality in the choice of $g_0$ implies that $u$ is
cyclically reduced and cyclically Dehn-reduced. Also, the minimality
in the choice of $h$ implies that the word $vuv^{-1}$ is freely
reduced as written. Since $v$ is a geodesic word, it is Dehn-reduced.
We claim that the word $vuv^{-1}$ is $\lambda$-reduced. Indeed,
suppose not. Then $vuv^{-1}$ contains a subword $w$ such that $w$ is
also a subword of some $r\in R$ with $|w|\ge (1-3\lambda)|r|$.

Note that since $u$ and $v$ are Dehn-reduced, the subword $w$ overlaps
at least two of the subwords $v,u,v^{-1}$ in $vuv^{-1}$.

{\bf Case 1.} The word $w$ is a subword of $vu$ or of $uv^{-1}$.
We assume that $w$ is a subword of $vu$, as the other case is
symmetric. Then $w=w_1w_2$, where $w_1$ is a terminal segment of $v$
and $w_2$ is an initial segment of $u$. We write $v$ and $u$ as
$v=v'w_1$ and $u=w_2u'$. Let $y\in F(A)$ be such that $r=wy=w_1w_2y$
in $F(A)$,
so that $|y|< 3\lambda|r|$. Note that since $u$ and $v$ are
Dehn-reduced and $|w|\ge (1-3\lambda)|r|$, it follows that
$|w_1|,|w_2|\ge (\frac{1}{2}-3\lambda)|r|\ge |r|/3> 3\lambda|r|$.
Observe also that $v=v'w_1=_Gv'y^{-1}w_2^{-1}$. Therefore
\begin{gather*}
g=_G vuv^{-1}=(v'w_1)(w_2u')(w_1^{-1}(v')^{-1})=_G\\
=_G v'y^{-1}w_2^{-1}w_2u'w_2y(v')^{-1}=(v'y^{-1}) (u'w_2)(y(v')^{-1}).
\end{gather*}
Since $|y|< 3\lambda|r|$ and $|w_1|>3\lambda|r|$, it follows that
$|v'y^{-1}|_A<|v|$, contradicting the minimality in the choice of $h$.

{\bf Case 2.} The subword $w$ overlaps  both $v$ and $v^{-1}$ in $vuv^{-1}$

If the overlap of $w$ with one of
$v,v^{-1}$ has length $\le \lambda|r|$, then either $vu$ or $uv^{-1}$
contains a subword of $r$ of length $\ge (1-5\lambda)|r|$, and we get
a contradiction similarly to Case 1.

If the overlaps of $w$ with each of
$v,v^{-1}$ have length $>\lambda|r|$, we get a contradiction with
the $C'(\lambda)$-small cancellation condition.

Thus part (1), (2), (3) and (4) of the lemma are established. Since the words
$vuv^{-1}$ and $z$ are both $\lambda$-reduced, part (5) of the lemma now follows
from Proposition~\ref{prop:eq-conj}.
\end{proof}

\section{Genericity}\label{sect:gen}

In this paper we work with the Arzhantseva-Ol'shanskii model of
genericity in free groups (see~\cite{AO,KMSS,KS}) based on the asymptotic
density considerations.

\begin{conv}
Let $F=F(A)$, where
$A=\{a_1,\dots, a_k\}$ and $k\ge 2$. Let $m\ge 1$ and let $U\subseteq
F^m$ be a subset of $F^m$. For $n\ge 0$ we denote by $\gamma_A(n,U)$ the
number of all $m$-tuples $(u_1,\dots, u_m)\in U$ such that $|u_i|_A=n$
for $i=1,\dots, m$. Note that for $n\ge 1$ we have
$\gamma_A(n,F^m)=\left(2k(2k-1)^{n-1}\right)^m$. We say that
$U\subseteq F^m$ is \emph{spherically homogeneous} if for every
$(u_1,\dots, u_m)\in U$ we have $|u_1|_A=\dots =|u_m|_A$. Note that
this restriction is vacuous if $m=1$. Let $S_m=S_{m,A}$ denote the set of all
$m$-tuples $(u_1,\dots, u_m)\in F^m$ such that $|u_1|_A=\dots
=|u_m|_A$.

Let $\mathcal C=\mathcal C_A$ be the set of all cyclically reduced words in $F(A)$.
Let $\mathcal C_m=\mathcal C_{m,A}=\mathcal C^m\cap S_m$. Thus $\mathcal C_m$ consists of all $m$-tuples $(u_1,\dots, u_m)$ of cyclically reduced words in $F(A)$ such that $|u_1|=\dots =|u_m|$.
\end{conv}

Although the notion of genericity make sense for arbitrary subsets of
$F^m$, for reasons of simplicity we will restrict ourselves to
spherically homogeneous subsets in this paper. Moreover, in applications we will only be concerned with tuples of cyclically reduced words.

\begin{defn} Let $F=F(A)$, where $A=\{a_1,\dots, a_k\}$ and $k\ge
  2$. Let $m\ge 1$. Let $\Omega\subseteq S_m$ be a spherically homogeneous subset. Let $U\subseteq \Omega$.

\begin{enumerate}
\item We say that \emph{$U$ is generic in $\Omega$} if
  \[
\lim_{n\to\infty}\frac{\gamma_A(n,U)}{\gamma_A(n,\Omega)}=1.
\]
If, in addition, the convergence to $1$ in the above limit is
exponentially fast, we say that $U$ is \emph{exponentially generic
  in $\Omega$}.

\item We say that a subset $U\subseteq \Omega$ is
  \emph{negligible in $\Omega$} (correspondingly \emph{exponentially
    negligible in $\Omega$}) if $\Omega\setminus U$ is generic (correspondingly
  exponentially generic) in $\Omega$.
\end{enumerate}

We stress that the above notions of genericity and negligibility are
highly dependent on the choice of a free basis $A$ of $F$. Therefore
in all our discussions regarding genericity such a free basis is
assumed to be fixed.

\end{defn}

The following is a straightforward corollary of the definitions:
\begin{lem}
Let $U\subseteq \mathcal C_m$.
Then the following are equivalent:
\begin{enumerate}
\item The set $U$ is exponentially negligible in $\mathcal C_m$.
\item We have
  \[
\lim_{n\to\infty}\frac{\gamma_A(n,U)}{(2k-1)^{nm}}=0,
\]
with exponentially fast convergence.
\item We have
\[
\limsup_{n\to\infty} \frac{\log \gamma_A(n,U)}{nm}<\log(2k-1).
\]
\end{enumerate}

\end{lem}

Similarly, the definitions imply:

\begin{lem} Let $m\ge 1$.
\begin{enumerate}
\item The union of a finite number of (exponentially) negligible
  subsets of $\mathcal C_m$ is (exponentially)
  negligible in $\mathcal C_m$.
\item The intersection of a finite number of (exponentially) generic
  subsets of $\mathcal C_m$ is (exponentially) generic
  in $\mathcal C_m$.
\item Let $U\subseteq \mathcal C$ be an exponentially generic subset. Then
  $U^m\cap S_m$ is exponentially generic in $\mathcal C_m$.
\end{enumerate}
\end{lem}

\begin{conv}
We say that a certain property of $m$-tuples of cyclically reduced elements of $F=F(A)$
is \emph{generic} (correspondingly, \emph{exponentially generic}) if there exists a
generic (correspondingly, exponentially generic) subset $U\subseteq
\mathcal C_m$ such that every $m$-tuple in $U$ has the property in question.
\end{conv}

We list some properties of $m$-tuples of cyclically reduced elements
of $F$ that are known to be
exponentially generic (see~\cite{AO,KS,KKS}).

\begin{prop}\label{prop:gen}
Let $F=F(A)$, where $A=\{a_1,\dots, a_k\}$ and $k\ge 2$.
  Let $m\ge 1$. Then:
\begin{enumerate}
\item The property that no element of an $m$-tuple is a proper power
  in $F$ is exponentially generic in $\mathcal C_m$.
\item Let $0<\lambda<1$ be arbitrary. Then the property that an $m$-tuple
  $(u_1,\dots, u_m)$, after cyclic reduction and symmetrization,
  satisfies the $C'(\lambda)$ small cancellation condition, is
  exponentially generic in $\mathcal C_m$.
\item The property that for an $m$-tuple
  $(u_1,\dots, u_m)$ for every $i\ne j$ the element $u_i$ is not conjugate
  to $u_j^{\pm 1}$ in $F$, is exponentially generic in $\mathcal C_m$.
\item Let $K\ge 1$ be an integer and let $0<\lambda<1$. Then the
  property of an  $m$-tuple $(u_1,\dots, u_m)$ that every
  subword $u$ of some $u_i$ of length $\ge \lambda |u_i|$ contains as a
  subword every freely reduced word of length $\le K$ in $F(A)$, is
  exponentially generic in $\mathcal C_m$.
\end{enumerate}
\end{prop}

\section{Stallings folds and Nielsen equivalence}\label{stallings}

In this section we review Stallings folds and prove a weak version of our main theorem. The proof of this special case is quite easy but still illustrates some of the ingredients of the proof of the general case.

\smallskip In a beautiful article \cite{St} Stallings used graphs to
represent subgroups of free groups and described how to use simple
operations called \emph{folds} to transform the graph into a graph at which a
basis of the subgroup can be read off. We briefly discuss Stallings folds for
the free group $F=F(a_1,\ldots,a_k)$ and refer the reader to \cite{KM}
for more detailed background information.

\smallskip

Put $A=\{a_1,\dots, a_k\}$. Let $R_A$ be the directed labeled graph consisting of a single vertex $v_0$ and $k$ loop-edges $e_1,\ldots ,e_k$ with labels $a_1,\ldots,a_k$, respectively. There is an obvious isomorphism $$\phi:\pi_1(R_A,v_0)\to F$$  that maps  the homotopy class represented by the loop edge $e_i$, travelled in positive direction, to $a_i$.

	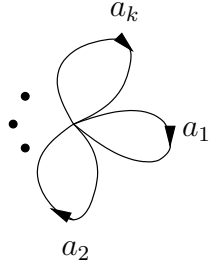
\begin{figure}[htb]
		\input{ra.pstex_t}
		\caption{The graph $R_A$}
	\end{figure}

\smallskip An  \emph{$A$-graph} is a directed labeled graph $\Gamma$
with edge-labels from $A$. For any $A$-graph $\Gamma$ there exists a unique label-preserving graph-map $p:\Gamma\to R_A$. After choosing a base-vertex $x$ of $\Gamma$, the morphism $p$ induces a homomorphism $p_*:\pi_1(\Gamma,x)\to \pi_1(R_A,v_0)$.

\smallskip Given a  tuple $T=(g_1,\ldots,g_s)$ of elements from
$\pi_1(R_A,v_0)=F$ we can construct a graph $S_T$ with base vertex $x_0$ such that $p_*(\pi_1(S_T,x_0))=\langle g_1,\ldots ,g_s\rangle$ as
follows. We assume that $g_i\neq 1$ for $1\le i\le l\le s$ and
$g_i= 1$ for $l<i\le s$.

\begin{enumerate}
\item $S_T$ is a wedge of $l$ circles with base vertex $x_0$ where the  $i$th circle is of simplicial length $|g_i|_A$.
\item For $i=1,\dots, l$, the label of the $i$-th circle is the reduced word representing $g_i$.
\end{enumerate}

	\begin{figure}[htb]
		\input{foldinginitial.pstex_t}
		\caption{The map $p$ from $S_T$ to $R_A$}
	\end{figure}
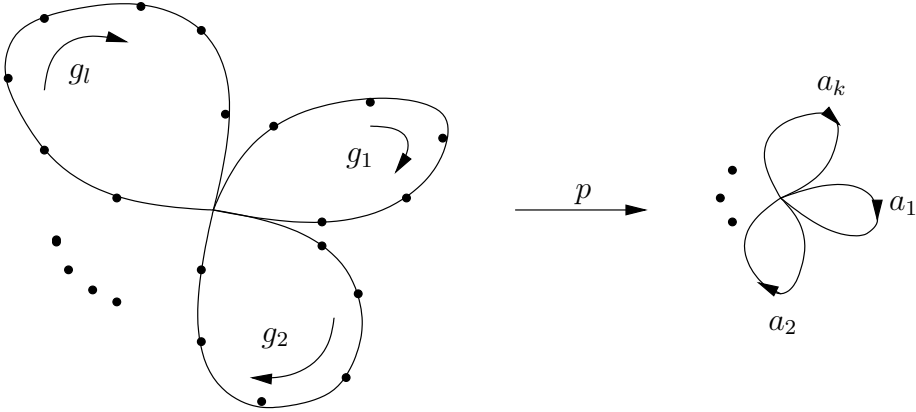

Stallings observed that any such $A$-graph $S_T$ (and, more generally,
any finite connected $A$-graph) can be modified by folds without
changing the image of the induced homomorphism so that the resulting
graph is \emph{folded}, i.e. the associated morphism $p$ is an
immersion. A \emph{fold} here is the identification of two edges that
have the same label and the same initial vertex or the same terminal
vertex. Note that one or both edges can be loop-edges. 

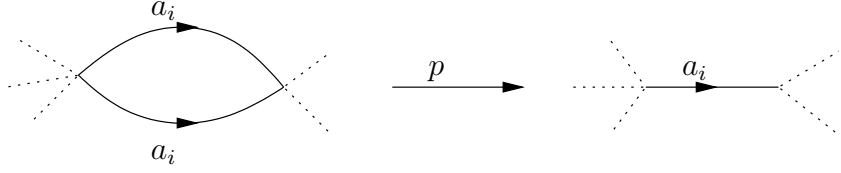
\begin{figure}[htb]{\input{def-fold.pstex_t}}
		\caption{A fold identifying two edges labeled $a_i$}
	\end{figure}
	
	If $\Gamma$ is a folded $A$-graph, then $p:\Gamma\to R_A$ maps reduced edge-paths to reduced paths which implies that $p_*$ is injective. 

\begin{lem}\label{lem:stallings} Let $\Gamma$ be a finite connected $A$-graph and
  $x_0$ be a vertex of $\Gamma$ such that
  $p_*:\pi_1(\Gamma,x_0)\to\pi_1(R_A,v_0)$ is surjective. Then there
  exists a finite sequence of $A$-graphs 
\[\Gamma=\Gamma_0,\Gamma_1,\ldots ,\Gamma_n=R_A
\] 
such that $\Gamma_{i}$  can be obtained from $\Gamma_{i-1}$ by a fold for $1\le i\le n$.
\end{lem}

\begin{proof} 
Note that each fold reduces the number of edges in a finite $A$-graph.
Thus any sequence of folds, that starts with the graph $\Gamma$, must
terminate. Let $\Gamma_n$ be obtained from $\Gamma$ by a maximal
sequence of folds, so that $\Gamma_n$ is folded. We claim that
$\Gamma_n=R_A$.

As $\Gamma_n$ is folded, it follows that $p_*$ is injective for any
choice of base point in $\Gamma_n$. Let $x_0^n$ the vertex of $\Gamma_n$ be the image of $x_0$ under the folds. By assumption $p_*:\pi_1(\Gamma_n,x_0^n)\to \pi_1(R_A,v_0)$ is surjective. In particular, for each $a_i\in A$, there must exist a loop edge at $x_0^n$ labeled with $a_i$, as the homotopy class represented by $e_i$ must be the image of some reduced path in $\Gamma_n$. The assumption that $\Gamma_n$ is folded now easily implies that $\Gamma_n=R_A$.
\end{proof}

\begin{defn}
We say that a finite connected graph $\Gamma$ is a \emph{core graph},
if it does not have any degree-one vertices. For a finite connected
graph $\Gamma$ with a nontrivial fundamental group let $Core(\Gamma)$
be the unique smallest subgraph of $\Gamma$ whose inclusion into
$\Gamma$ is a homotopy equivalence.
\end{defn}

Thus $Core(\Gamma)$ is obtained
from $\Gamma$ by cutting off a (possibly empty) collection of
tree-branches. Note that $Core(\Gamma)$ is a core graph.
We need the following simple observation about $A$-graphs that fold
onto $R_A$ with a single fold. 

Let $\Gamma$ be an $A$-graph. We say that an $A$-graph $\Gamma'$ is
\emph{$\Gamma$ with a spike-edge}, if $\Gamma'$ is obtained from $\Gamma$ by adding an edge $e$ such that the initial vertex lies in $\Gamma$ and the terminal vertex does not.

\begin{lem}\label{trivial} Let $\Gamma$ be a finite core $A$-graph of
  rank $k$.
\begin{enumerate}
\item If $\Gamma$  folds onto $R_A$ with a single fold then there exists
  a reduced word of length 2 that cannot be read as the label of an
  edge-path in $\Gamma$.
\item If $\Gamma$  folds with a single fold onto $R_A$ with a spike-edge, then there exists
  a reduced word of length 4 that cannot be read as the label of an
  edge-path in $\Gamma$.

\end{enumerate}
\end{lem}

\begin{proof}

(1) Note that the fold must identify a loop edge with a non-loop edge
since identifying two loop edges decreases the rank and identifying
two non-loop edges yields a graph with a non-loop edge.  Thus
w.l.o.g. we can assume that $\Gamma$ has two distinct vertices $x$ and
$y$ and the fold identifies a loop edge at $x$ and an edge from $x$ to
$y$ both with label $a_1$. The graph $\Gamma$ has $k-1$ more edges that are labeled by $a_2,\ldots ,a_k$.

	\begin{figure}[htb]
		\input{corefoldsontora.pstex_t}
		\caption{A core graph that folds onto $R_A$ with a single fold}
		\label{fold1}
	\end{figure}
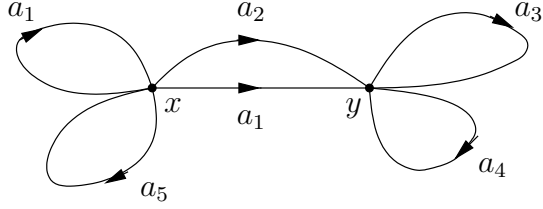

As we assume $\Gamma$ to be a core graph, there must either exist a loop edge at $y$ or a second edge from $x$ and $y$. W.l.o.g. we can assume that this edge is labeled by $a_2$, see Figure~\ref{fold1}. In either case, we cannot read the word $a_2a_1$ in $\Gamma$ as the path is at the vertex $y$ after reading $a_2$; this proves the assertion.
\smallskip

(2) In this case the fold must identify two non-loop edge with distinct endpoints, i.e. we must be in the situation of Figure~\ref{fold2}. If a reduced word is read by a path in $\Gamma$ then the vertex at which the fold is based can only occur as the initial and the terminal vertex of this path. As there is clearly a word of length 2 that cannot be read in the graph spanned by the remaining vertices (the word is $a_1a_1$ in the example), the claim follows.
\end{proof}
	\begin{figure}[htb]
		\input{corefoldsontora+.pstex_t}
		\caption{A core graph that folds onto $R_A$ with a spike-edge}
		\label{fold2}
	\end{figure}
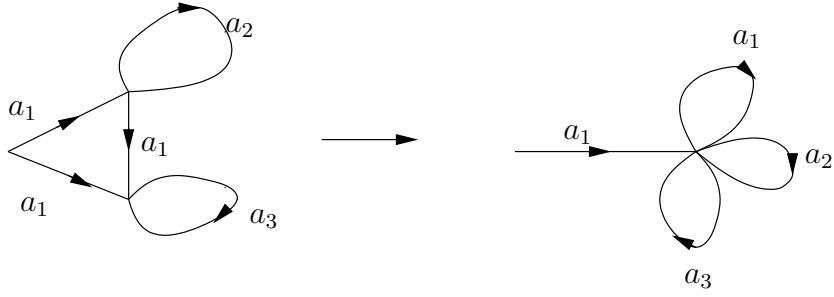

We are now able to prove a weak version of our main theorem. This weak
version is an easy consequence of Stalling folds and the above
observations. However, one of the key ideas used in the proof of the
main theorem already shows up in the proof of the weak version.

\begin{thm}\label{specialcase} Let $k\ge m\ge 2$ and let
  $A=\{a_1,\dots, a_k\}$ and $B=\{b_1,\dots, b_m\}$ be free bases of
  $F(A)$ and $F(B)$. Then there exist exponentially generic subsets $V\subseteq
  \mathcal C_{m,A}$ and $U\subseteq \mathcal C_{k,B}$ with the following properties.

Let $(u_1(\bar b),\dots,u_k(\bar b))\in U$ and $(v_1(\bar a), \dots, v_m(\bar a))\in V$
be arbitrary and let $G$ be given by the presentation
\[
G=\langle a_1,\ldots,a_k,b_1,\ldots b_m\,|\,a_i=u_i(\bar b),
b_j=v_j(\bar a)\hbox{ for }1\le i\le k, 1\le j\le m\rangle.\tag{!}
\]

Then the $k$-tuple
 $(a_1,\ldots, a_k)$ is not Nielsen equivalent in $G$ to a $k$-tuple of form $(b_1,*,\ldots,*)$.
\end{thm}

\begin{proof}

We can choose exponentially generic sets $V\subseteq
  \mathcal C_{m,A}$ and $U\subseteq \mathcal C_{k,B}$ so that the following hold:
\begin{enumerate}
\item For every $(u_1(\bar b),\dots,u_k(\bar b))\in U$ and $(v_1(\bar
  a), \dots, v_m(\bar a))\in V$ the presentation (!) satisfies the
  $C'(1/100)$ small cancellation condition, after symmetrization.
\item For every $(u_1(\bar b),\dots,u_k(\bar b))\in U$ and $(v_1(\bar
  a), \dots, v_m(\bar a))\in V$ we have $|u_i|_B\ge 10^{10}$ and
  $|v_j|_A\ge 10^{10}$.
\item For every $(v_1(\bar
  a), \dots, v_m(\bar a))\in V$ and every $j=1,\dots, m$, every subword
  of $v_j$ of length $\ge |v_j|/100$ contains all freely reduced words
  of length two in $F(A)$ as subwords.
\end{enumerate}

We now argue by contradiction. Assume that for some $G$ as in
Theorem~\ref{specialcase}, the $k$-tuple $(a_1,\ldots, a_k)$ is
Nielsen equivalent in $G$ to a $k$-tuple $(b_1,*,\ldots,*)$. This
implies that $(a_1,\ldots,a_k)$ is
Nielsen equivalent in $F(A)$ to a tuple $T=(w_1(\bar a),\ldots,w_k(\bar a))$ such that
$w_1(\bar a)=b_1$ in $G$. Note that we might have $w_1(\bar a)\neq
v_1(\bar a)$ in $F(A)$. We do not distinguish between the $w_i(\bar
a)$ and freely reduced words in $F(A)$ representing them. After a
possible conjugation of $T$ in $F(A)$, we may assume that $w_1$ is
cyclically reduced in $F(A)$ and that $w_1$ is conjugate to $b_1$ in $G$.

\smallskip By Corollary~\ref{cor:conj-sc}, $w_1(\bar a)$ must contain
at least one fourth of a cyclic permutation of one of the defining
relations $\left(v_j(\bar a)b_j^{-1}\right)^{\pm 1}$, since $w_1(\bar a)$ is
cyclically reduced and is conjugate to $b_1$ in $G$.
This implies in particular that $w_1$ contains every freely reduced
 word of length two in $F(A)$ as a subword. 
\smallskip Let now $S_T$ be the wedge as above and choose a sequence
$$S_T=\Gamma_0,\Gamma_1,\ldots,\Gamma_n=R_A$$ such that each $\Gamma_i$ can
be obtained from $\Gamma_{i-1}$ by a fold. Note that each $\Gamma_i$
is a core graph. Indeed, $\Gamma_i$ is obtained by a sequence of folds
from $\Gamma_0=S_T$, which was a wedge of loops labeled by freely
reduced words, and hence no vertex, other than possibly the
base-vertex, has degree $1$ in $\Gamma_i$. Moreover, the cyclically
reduced word $w_1$ can be read $\Gamma_i$ along a closed path based at
the base-vertex, which implies that the base-vertex has degree $>1$ in
$\Gamma_i$. Thus, indeed, $\Gamma_i$ is a core graph. 
 Hence $\Gamma_{n-1}$ is a core graph that folds onto $R_A$ with a single fold and $w_1$ can be read by a closed path in $\Gamma_i$. This contradicts Lemma~\ref{trivial}.
\end{proof}

\begin{rem} The proof of Theorem~\ref{specialcase} recovers the well-known fact that for any cyclically reduced primitive element $g$ in $F(a_1,\ldots,a_k)$ there exists some reduced word $w$ of length two such that neither $w$ nor $w^{-1}$ occur as a subword of $g$.
\end{rem}

\smallskip We conclude this section with a generalization of Lemma~\ref{trivial} that will be needed in the proof of the main theorem.

\begin{lem}\label{not-so-trivial} Let $\Gamma$ be a finite connected
  $A$-graph of rank $k+1$ that folds onto $R_A$ but does not contain
  an isomorphic copy of $R_A$ as a subgraph. Then there exists a reduced word that is not a label of a path in $\Gamma$.
\end{lem}

\begin{proof} We may assume 
  that any applicable fold in $\Gamma$ produces a subgraph that is isomorphic to
  $R_A$ since we could otherwise just apply a fold that does not produce
  such a subgraph. As any word that was readable before the fold is
  also readable after the fold, it follows that the assertion for the
  new graph implies the assertion for the original one.

Since $\Gamma$ is obtained from $Core(\Gamma)$ by attaching a finite
  number of finite disjoint trees, any reduced word readable in
  $\Gamma$ has the form $z'zz''$ where $z$ is readable in
  $Core(\Gamma)$ and where $|z'|+|z''|\le {\rm diam}(\Gamma)$. Thus,
  if $w$ is a nontrivial freely reduced word such that $w^M$ is readable
  in $\Gamma$ with $M:=2\, {\rm diam}(\Gamma)+1$, then $w$ is readable in
  $Core(\Gamma)$. Hence any reduced word that is not readable
  in $Core(\Gamma)$ has a power that is not readable in $\Gamma$.
  We can therefore assume that $\Gamma=Core(\Gamma)$.

Thus $\Gamma$ contains a graph $\Psi$ which either folds
onto $R_A$ or onto $R_A$  with a spike-edge. Furthermore
$\Gamma$ is obtained from $\Psi$ by attaching a subdivided segment or a subdivided lollipop
as described in Lemma~\ref{coredescription}~(1).


There are two cases: Case 1, where $\Psi$ folds onto $R_A$, and Case
2, where $\Psi$ folds onto $R_A$ with a spike-edge. We will
give all details in Case 1 and leave Case 2 to the reader. No new
arguments are needed to complete the proof.

\smallskip \noindent {\bf Case 1: } the graph $\Psi$ folds onto $R_A$
with a single fold. Recall (see proof of Lemma~\ref{trivial}) that
w.l.o.g. $\Psi$ has two vertices $x$ and $y$, a loop edge at $x$ with
label $a_1$, an edge with label $a_1$ from $x$ to $y$ and $k-1$ more
edges with labels $a_2,\ldots ,a_k$, see Figure~\ref{fold1}. $\Gamma$
is obtained from $\Psi$ by attaching an segment, respectively a lollipop,  consisting of, say, $l$ edges.

\smallskip
In the following we will denote the attached segment, respectively the
attached lollipop, by $s$. By the label of $s$ we mean the label of
the segment if $s$ is an segment and otherwise the label of the path
that starts and ends at the vertex where $s$ is attached to $\Psi$ and
walks once around the lollipop. We can assume that the label of $s$ is
freely reduced  as we are only looking at freely reduced words that
can be read in $\Gamma$. We distinguish three subcases depending on
the size of $l$, the number of edges of $s$.

\smallskip
\noindent {\bf Case 1A: } Suppose first that $l\ge 3$. Then we cannot fold an edge of $s$  with an edge of $\Psi$ as such a fold could not produce a copy of $R_A$ contradicting our assumption. Note next that this assumption implies that the label of $s$ cannot be a power of $a_1$ as we could otherwise apply a fold to $\Gamma$ without producing a subgraph isomorphic to $R_A$. The same argument shows, that if $s$ starts (ends) at $x$ then the label of $s$ does not start (end) with the letter $a_1^{\pm 1}$.

Then any path that reads $a_1^{-l}$ must either end  at the vertex $x$ or $l$ must be a lollipop with non-degenerate stick and the path travels exclusively on the candy of the lollipop which itself is labeled by a power of $a_1$, see Figure~\ref{fig:case1a}.

	\begin{figure}[htb]
		\input{case1a.pstex_t}
		\caption{The label of the lollipop $s$ is $a_5a_1a_1a_5^{-1}$. The word $a_1^l$ can be read in $\Psi$ and also on the lollipop.}
		\label{fig:case1a}
	\end{figure}
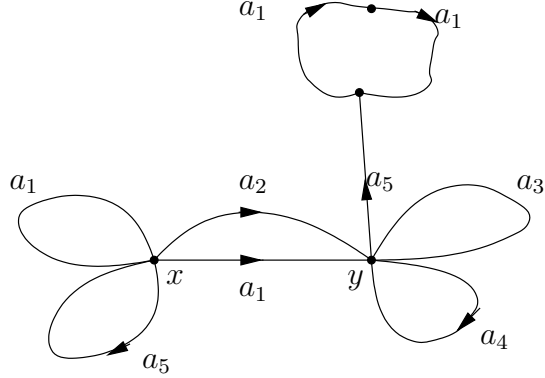

If we now pick $a_j^{\eta}\neq a_1^{\pm 1}$ such that $a_j^{\eta}$ is not the last edge on the stick of  lollipop if the lollipop has a non-degenerate stick (choose $a_j^{\eta}\neq a_5$ in the example drawn in Figure~\ref{fig:case1a}), then any path that reads $a_j^{\eta}a_1^{-l}$ must end at vertex $x$.

Now pick a letter $a_k^{\varepsilon}$ such there is no edge of $\Psi$ with initial vertex $x$ and label $a_k^{\varepsilon}$. It follows that a path reading the word $a_j^{\eta}a_1^{-l}a_k^{\varepsilon}$ must terminate with the initial edge of $s$ or $s^{-1}$.
It follows that the word $a_j^\eta a_1^{-l}a_k^{\varepsilon}a_l^{\xi}$ cannot be read in $\Gamma$ if $a_l^{\xi}$ is such that $a_k^{\varepsilon}a_l^{\xi}$ is not the initial word of the label of $s$ (if $s$ starts at $x$) and not the initial word of the label of $s^{-1}$ if $s$ ends at $x$. This proves the claim as it implies that $a_j^{\eta}a_1^{-l}a_k^{\varepsilon}a_l^\xi$ is not readable in $\Gamma$.

\medskip \noindent {\bf Case 1B:} Suppose now that $l=2$. If $s$ does not fold onto $\Psi$ then we argue as in the case $l\ge 3$. Thus we can assume that one of the two edges of $s$ folds onto an edge of $\Psi$. As such a fold does not identify $x$ and $y$ it follows that this fold must produce a new loop edge attached to either $x$ or $y$ and that there must have been already a wedge of $k-1$ circles at $x$ or $y$. We distinguish three different configurations.

\smallskip Suppose first that $s$ is a segment  of length $2$ from $x$ to $y$. In this case an edge of $s$ must be folded onto an edge from $x$ to $y$ to produce a loop.

If the fold produces a copy of $R_A$ based at $y$ then $\Psi$ must have consisted of a single loop with label $a_1$ at $x$, a single edge with label $a_1$ from $x$ to $y$ and $k-1$ loop edges with labels $a_2,\ldots ,a_k$ at $y$. The fold must add a loop with label $a_1$ which implies that the segment $s$ from $x$ to $y$ has label $a_1a_1$. This means that the path $a_1^{-1}a_2$ is not readable in $\Gamma$ as no path is at the vertex $y$ after reading $a^{-1}$, see Figure~\ref{fig:case1b1}.

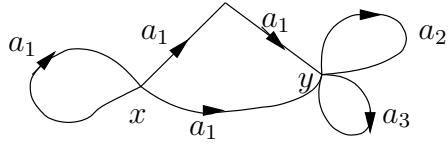
\begin{figure}[htb]
		\input{cases1b1.pstex_t}
		\caption{A copy of $R_A$ emerges at $y$}
		\label{fig:case1b1}
	\end{figure}

If the fold produces a copy of $R_A$ based at $x$ then $\Psi$ had
w.l.o.g. $k-1$ loop edges based at $x$ with labels $a_1,a_3,\ldots
,a_k$, an edge with label $a_1$ from $x$ to $y$ and either (a) a loop
edge with label $a_2$ based at $y$ or (b) an edge with label
$a_2^{\eta}$ from $x$ to $y$. The fold must produce a loop edge with
label $a_2$ at $x$ which implies that the segment $s$ from $x$ to $y$
must have label  $a_2^\eta a_1$ in case (a) or $a_2^{2\eta}$ or
$a_2^{\eta}a_1$ in case (b). In all of these cases it is clear that
$a_2^{2\eta}a_1^{-1}$ is not readable in $\Gamma$, see Figure~\ref{fig:case1b2}.

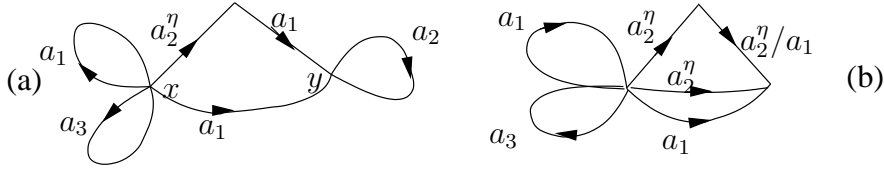
\begin{figure}[htb]
		\input{cases1b2.pstex_t}
		\caption{A copy of $R_A$ emerges at $x$}
		\label{fig:case1b2}
	\end{figure}

Suppose now that $s$ is a loop of length $2$. As the fold must produce a loop-edge at $x$ or $y$, it must fold an edge of $s$ onto a loop edge. The descriptions of $\Psi$ are as in the case before.

If a loop edge is added at $y$ then the added loop-edge must have
label $a_1$, and hence $s$ must have been a loop based at $y$ with label $a_i^{\pm 1}a_1^{\pm 1}$ with $i\ge 2$. In this case the word $a_1^{-2}a_2$ is not readable in $\Gamma$.

If a loop edge is added at $x$ then the added loop edge must have label $a_2$ which implies that $s$ must be a loop based at $x$ with label $a_i^{\pm 1}a_2^{\pm 1}$ with $i\neq 2$. It follows that $a_2^2a_1^{-1}$ is not readable in $\Gamma$.

\smallskip Suppose lastly that $s$ is a lollipop whose stick and loop both have a single edge. In this case we can argue precisely as in Case 1A.

\smallskip \noindent {\bf Case 1C: } Suppose now  that $l=1$. We distinguish the cases where $s$ is a loop edge and where it is an edge from $x$ to $y$.

Suppose first that $s$ is a loop edge. If we add a loop edge with
label $a_j$ to $x$ then there must still be a reduced word of the form $a_1^{-1}a_i^{\pm
  1}$ (with $i\ne 1$) that cannot be read in $\Gamma$.

Indeed, a path with label $a_1^{-1}a_i^{\pm 1}$ must be at the vertex
$x$ after reading $a_1^{-1}$. Since by assumption $\Gamma$ contains no
copy of $R_A$, there is some $ 2\le i\le k$ such that there is no loop
edge labelled by $a_i$ based at $x$. Moreover there are not two non-loop edges with labels $a_i$ and $a_i^{-1}$ originating at $x$. The claim follows.

The same argument works if we add a loop edge with label different from $a_1$ at $y$. If we add a loop edge with label $a_1$ to $y$ then inspecting the two possibilities reveals that there always exists a word of type $a_2^\eta a_j^\epsilon $ that is not readable as we assume that not all loop edges are present at~$y$.

Suppose lastly that $s$ is an edge from $x$ to $y$. If the edge is
labeled by $a_1^{\pm 1}$ then there is a word of type $a_2^\eta
a_j^\varepsilon$ that is not readable in $\Gamma$. Thus we can assume
that the edge $s$ is labeled by an element different from $a_1^{\pm 1}$. This implies that any path that starts with $a_1^{-1}$ is at the vertex $x$ after reading $a_1^{-1}$. Thus we can assume that any $a_i^{\pm 1}$ can be read by an edge emanating at $x$. For at least one $a_j^{\varepsilon}$ the corresponding edge leads to $y$. Hence the word $a_1^{-1}a_j^{\varepsilon}a_1$ cannot be read in $\Gamma$.
\end{proof}

\section{Crucial consequences of genericity}

\begin{defn}[Critical words]
Let $A=\{a_1,\dots, a_k\}$ where $k\ge 2$.
In the following we will call an $A$-graph \emph{critical} if it
folds with a single fold onto a the rose $R_A$ or on $R_A$ with a spike-edge. We further call a word in $F(A)$ {\em critical} if it can be
read in some critical $A$-graph.
\end{defn}
By the lemmas established in the previous section, critical words are
highly non-generic in $F(A)$.

For the remainder of the paper we fix $k\ge 2$, $m\ge 2$ and the
alphabets $A=\{a_1,\dots, a_k\}$, $B=\{b_1,\dots, b_m\}$.
In this section we will study elements in a group of type \[
G=\langle a_1,\ldots,a_k,b_1,\ldots b_m\,|\,a_i=u_i(\bar b),
b_j=v_j(\bar a)\hbox{ for }1\le i\le k, 1\le j\le m\rangle.\tag{!!}
\]  that can be written as products (in $G$) of the form 
\[
w_0h^{\varepsilon_1}w_1h^{\varepsilon_2}w_2\cdot\ldots\cdot
w_{k-1}h^{\varepsilon_k}w_k \tag{$\ast\ast$}
\]
 where the $w_i$ are critical words, $h$ is some fixed element of $G$ and $\varepsilon_{i}\in\{-1,1\}$ for all $i$. We will see that there exist exponentially generic subsets $V\subseteq
  \mathcal C_{m,A}$ and $U\subseteq \mathcal C_{k,B}$ such that for  $(u_1(\bar b),\dots,u_k(\bar b))\in U$ and $(v_1(\bar a), \dots, v_m(\bar a))\in V$ these products are well-behaved. The results we obtain are the main tools for the proof of the main theorem. 
We will denote by $R_\ast\subseteq F(A\cup B)$ the symmetrized closure of the set of
defining relations in (!!). 

\smallskip We will not specify at the outset an explicit list of conditions for
$V$ and $U$ that are exponentially generic and that will ensure that the subsequent results hold. Rather we will make the
several required choices for $V$ and $U$ along the way in the proof,
and each of $U$, $V$ will be constructed as the intersection of a
finite number of exponentially generic sets.

\smallskip However, at the start we do require $V$ and $U$ to satisfy the
following conditions. We choose $0\le \lambda\le 10^{-100}$ and
$0<\lambda_1\le 10^{-100}\lambda$.

\begin{cond}\label{cond:gen}
The following hold for $U$ and $V$:
\begin{enumerate}
\item There are exponentially generic subsets $\mathcal V\subseteq \mathcal C_A$ and $\mathcal U\subseteq \mathcal C_B$ such that
$U\subseteq \mathcal U^k$ and $V\subseteq \mathcal V^m$.
\item For every $(u_1(\bar b),\dots,u_k(\bar b))\in U$ and $(v_1(\bar
  a), \dots, v_m(\bar a))\in V$ the set $R_\ast$ satisfies the
  $C'(\lambda_1)$ small cancellation condition, after symmetrization.
\item For every $u\in \mathcal U$ and $v\in \mathcal V$ we have $|u|_B\ge
  10^{100000}/\lambda_1$ and $|v|_A\ge 10^{100000}/\lambda_1$.
\item For every $v\in \mathcal V$, every subword of $v$ of length $\ge
  \lambda_1(|v|+1)/1000$ contains all freely reduced words of length
  $1000$ in $F(A)$ as subwords.
\item For every $u\in \mathcal U$, every subword of $u$ of length $\ge
  \lambda_1(|u|+1)/1000$ contains all freely reduced words of length
  $1000$ in $F(B)$ as subwords.
\item If some defining relation $r\in R_\ast$ contains two distinct
  occurrences of a subword $y$ then $|y|\le \lambda_1|r|/1000$.
\end{enumerate}
\end{cond}

Proposition~\ref{prop:gen} implies that exponentially generic subsets
$V\subseteq \mathcal C_{m,A}$ and $U\subseteq \mathcal C_{k,B}$ with the above
properties exist.

\smallskip As there are only finitely many isomoprphism classes of
critical $A$-graphs, and for any such $A$-graph $\Psi$ there exists a
word in $F(A)$ that is not readable in $\Psi$, Lemma~\ref{trivial} implies that
we may also assume the following:

\begin{cond}\label{cond:gen1} 
If $y\in F(A)$ is a freely reduced word that is readable in
some critical $A$-graph such that $y$ is also a subword of some
defining relation $r\in R_\ast$
then $|y|\le \lambda_1|r|/1000$.
\end{cond}

\begin{defn}[Exceptional words]
We say that a word $h$ in $F(A)$ is \emph{exceptional}, if it does not
contain a subword of length $N/100$ that is also a subword of some $r\in R_\ast$.

We say that a reduced word $h$ in  $F(A\cup B)$ is a
\emph{non-exceptional} if $h$ is not exceptional, that is, if either $h$ contains some $b_i$ or $h\in
F(A)$ and $h$ contains a subword of length $N/100$
that is also a subword of some $r\in R_\ast$. 
\end{defn}

We will mostly be interested in products as in $(\ast\ast)$ above
where $w_i\in F(A)$ are critical and where $h\in F(A\cup B)$ is
non-exceptional. To understand the possible cancellations in such
products we need a series of lemmas.

\begin{lem}\label{cancel2}
Let $h$ a non-exceptional word in $F(A\cup B)$.

For any freely reduced critical word $t\in F(A)$  let $W^t$ be the
freely reduced form of the word $hth$. We write $W^t$ as a reduced
product $W^t=W_1^tW_2^tW_3^t$, where $W_1^t$, $W_2^t$ and $W_3^t$ are
the (possibly empty) portions of  $h$, $t$, $h$ accordingly that
survive the free reduction of the product $hth$.

\smallskip There exists a critical word $w\in F(A)$ such that for each critical reduced word $\bar w\neq w$ the following hold:

\begin{enumerate}
\item The word $W^{\bar w}$ does not contain a subword
  $y$ of a defining relator $r\in R_\ast$ such that $y$ has overlaps with both $W_1^{\bar w}$ and $W_3^{\bar w}$ in subwords of length at least $100\lambda_1N$;
\item If in the free reduction of $h \bar wh $ some letter from $B^{\pm
  1}$ cancels then $h =ub_i^{\varepsilon}r_1\hat g r_2b_i^{-\varepsilon}v^{-1}$ such that the following hold:
\begin{enumerate}
\item $u$ and $v$ are words in $F(A)$.
\item $r_1$ and $r_2$ are words in $F(B)$ and $|r_1|,|r_2|\ge 100\lambda_1N-1$.
\item $W^w=ub_i^{\varepsilon}r_1\hat g
  r_2b_i^{-\varepsilon}a_j^{\eta}b_i^{\varepsilon}r_1\hat g
  r_2b_i^{-\varepsilon}v^{-1}$, with $W_1^w=ub_i^{\varepsilon}r_1\hat g
  r_2b_i^{-\varepsilon}$, $W_2^w=a_j^\eta$, $W_3^w=b_i^{\varepsilon}r_1\hat g
  r_2b_i^{-\varepsilon}v^{-1}$ and the word
  $r_2b_i^{-\varepsilon}a_j^{\eta}b_i^{\varepsilon}r_1$ is a subword
  of a defining relator from $R_\ast$, so that condition (1) above fails for the word
  $w$. 
\item $W^{\bar w}=ub_i^{\varepsilon}r_1\hat g \bar r_2\bar r_1\hat g
  r_2b_i^{-\varepsilon}v^{-1}$ where for $i=1,2$ $\bar r_i$ is the
  remnant of $r_i$ and $|\bar r_i|\ge 90\lambda _1 N$; thus $W_1^{\bar
    w}=ub_i^{\varepsilon}r_1\hat g \bar r_2$, $W_2^{\bar w}=1$ and  $W_3^{\bar w}=\bar r_1\hat g
  r_2b_i^{-\varepsilon}v^{-1}$. Furthermore $\bar r_2\bar r_1$ is not a subword of a defining relator.

\end{enumerate}
\item In the free reduction of $h \bar wh $ no subword $y$ of one
  of the two occurrences of $h $ cancels such that $y$ is also a subword
  of some defining relator in $R_\ast$ and $|y|\ge 100\lambda_1N$.
\end{enumerate}
Moreover there is at most one critical word $\bar w$ such that (2) occurs.
\end{lem}

\begin{proof}

We first show that there is at most one word $w$ such that the conclusion of either (1) or (3) fails. We can assume that $h$ contains a subword of a relation of length at least $100\lambda_1 N$ as there is nothing to show otherwise.

\smallskip Let $Y$ be the right-most subword  in $h $ such that $Y$ is
a subword of some defining relator $r\in R_\ast$ and that $|Y|=99\lambda_1N$. Similarly, let $Y'$ be the left-most
subword in $h $ such that $Y'$ is a subword of some defining relator
from $R_\ast$ such that  $|Y'|=99\lambda_1N$. 

\smallskip In what follows, when referring to $Y$ we think of it as a subword in the first $h $ in $h  w h $ and  when referring to $Y'$ we think of it as a subword in the second $h $ in $h  w h $. Recall that a critical subword of a defining relator has length at most $\lambda_1N/100$.

\smallskip If case (3) fails, then a subword of $Y$ of length at least $96\lambda_1N$ must cancel with the subword of $Y'$  of the same length in the free reduction of $h  w h $.
Therefore, if (3) fails for two distinct $w_1\ne w_2$, then either $Y$ has the form $Y=Y_1y_2=y_1Y_1$ with $0<|y_1|=|y_2|\le 4\lambda N$ or $Y'$ has the form  $Y'=Y_1'y_2'=y_1'Y_1'$ with $0<|y_1'|=|y_2'|\le 4\lambda N$. It follows that $Y$ or $Y'$ is a periodic word, contrary to our assumptions about generic properties of $\mathcal U$ and $\mathcal V$. Thus (3) can fail for at most one $w$.

\smallskip Suppose there exist distinct $w_1$ and $w_2$ as in the statement of the lemma such that (1) fails for $w_1$ and (3) fails for
$w_2$.
Again, in the free reduction of $h  w_2 h $, a subword of $Y$ of length at least $96\lambda_1N$ must cancel with the subword of $Y'$.
If in the free reduction of $h  w_1 h $ a subword of $Y$ or $Y'$ of length $\ge 70\lambda_1 N$ cancels, then we get a contradiction similarly to the above proof that (3) can fail for at most one $w$. Hence an initial segment $s$ of $Y$ of length $|s|\ge 29N\lambda_1$ and a terminal segment $s'$ of $Y'$ of length $|s'|\ge 29N\lambda_1$ survive the free reduction of $h  w_1 h $.   The choice of $Y$ and $Y'$ now implies that each of $s,s'$ is a subword of the word $y$, where $y$ is a subword  of some defining relator $r$.
Recall, however, that since (3) fails for $h  w_2h $, the words $Y$
and $Y'$ almost cancel with each other and hence they are almost
inverse. More precisely, a subword of $Y$ of length at least
$96\lambda_1N$ must cancel with the subword of $Y'$ of the same
length. This means that there is a subword $t$ of $s$ of length
$|q|\ge 10\lambda_1N$ such that $t^{-1}$ is a subword of $s'$. Thus
both $t$ and $t^{-1}$ are subwords of $r\in R_\ast$ which contradicts the small cancellation assumption about (!!).

\smallskip Suppose now that (1) fails for two different $w_1$ and $w_2$ and let $y_1,y_2,r_1,r_2$
be the corresponding subwords and defining relators as spelled out in (1).  
By the assumptions in
Condition~\ref{cond:gen1}, the word $W_2^{w_i}$ can have an overlap with
$r_i$ of length at most $\lambda_1|r_i|/1000$. Hence  $|W_2^{w_i}|\le \lambda_1|r_i|/1000$ for $i=1,2$.


We claim
that $\left| |W_i^{w_1}|-|W_i^{w_2}|\right|\le \lambda_1N/10$  for $i=1,3$.
Indeed, if that were not the case, then (3) would fail for either  $w_1$ or $w_2$. But we have already established that it
is not possible for (1) to fail for one $w$ and for
(3) to fail for a different $w$. Thus indeed $\left|
  |W_i^{w_1}|-|W_i^{w_2}|\right|\le \lambda_1N/10$  for $i=1,3$.
The fact that the subword $y_i$ of $W^{w_i}$ overlaps each of
$W_i^{w_1}$ and $W_3^{w_i}$ in subwords of length $\ge N/10$, together
with  the small cancellation assumptions, implies that $r_1$ and $r_2$
are cyclic permutations of the same relator $r\in R_\ast$. We now claim that
$W_j^{w_1}=W_j^{w_2}$ for $j=1,2,3$. If this were not the case, then $r$
would contain subwords of the form $\alpha\beta\gamma$ and
$\alpha\beta'\gamma$ where $|\alpha|,|\gamma|\ge \lambda|r|/100$,
where $|\beta|,|\beta'|\le \lambda_1|r_i|/1000$ and where $\beta\ne
\beta'$. This contradicts the small cancellation assumption on (!!).
Hence, indeed, (1) can fail  for at most one $w$ as in the lemma.

\smallskip We have shown that there exists at most one word $w$ such that either (1) or (3) fails. Note that there is at most one word $\bar w$ such that in $h \bar wh $ some letter $b_i^{\pm 1}$ cancels. If $w=\bar w$ or if $w$ or $\bar w$ with the above properties does not exist then there is nothing to show. Thus we can assume that there is a word $w$ such that either (1) or (3) fails and that for $\bar w$ some letter $b_i^{\pm 1}$ cancels. We have to show that this puts us into situation (2).

\smallskip Suppose that (3) fails for $w$.
Let $b$ be the right-most occurrence of a letter $b_i^{\pm 1}$ in the first $h $. Then the left-most occurrence of a
letter from $B^{\pm 1}$ in the second $h $ must be $b^{-1}$ since in $h  \bar w h $ these occurrences must cancel each other.
Suppose first that $b$ is to the left of the mid-point of $Y$ in $h $. Then $Y,Y'\in F(A)$.
Then all of $Y,Y'$ have to cancel in both $h  w h $ and in $h  \bar w h $. This is impossible by the same argument as in the proof that (3) can happen for at most one $w$. Thus $b$ is to the right of the midpoint of $Y$ in the first $h $.
Hence $b$ has to cancel both in $h  w h $ and $h  \bar w h $, which is impossible, as noted above.

\smallskip Suppose now that (1) fails for $w$. Again, let $b$ be the
right-most occurrence of a letter $b_i^\pm 1$ in the first $h $. Then
the left-most occurrence of a letter from $B^{\pm 1}$ in the second $h
$ must be $b^{-1}$ since in $h \bar w h $ these occurrences must
cancel each other. Since $w\ne \bar w$, these occurrences of $b$ and
$b^{-1}$ do not cancel in $h  w h $. Recall that the freely reduced
form $W^{w}=W_1^{w}W_2^{w}W_3^{w}$ of $h  w h $ contains a subword $y$
such that $y$ is also a subword of a defining relation $r\in R_\ast$ and that $y$ overlaps with both $W_1^{w}$ and $W_3^{w}$ in subwords of length at least $100\lambda_1N$.

Suppose first that, up to cyclic permutation and inversion, $r$ has
the form $a_i^{-1}u_i(\tilde b)$. It follows that exactly one letter,
namely $a_i^{\pm 1}$ survives in the free reduction of the segment
between $b$ and $b^{-1}$ in $h  w h $ and, moreover, both the
occurrences $b$ and $b^{-1}$ are inside of $y$. It follows that
$W_1^{w}$ ends with a subword $c_1$ and $W_3^{w}$ begins with a
subword $c_2$ which are of length at least $100\lambda_1N$ and which
are the terminal (initial) word of some $u_i(\tilde{b})^{\pm 1}$. Now
genericity of $u_i$ implies that not more than a subword of length
$10\lambda_1 N$ can cancel in $c_1c_2$ which puts us into (2).

Suppose lastly that up to a cyclic permutation and inversion, $r$ has
the form $b_i^{-1}v_i(\tilde a)$. If $b$ occurs to the left of the
mid-point of $Y$ in $h $ then more than half of $Y$ cancels in $h
\bar w h $, which yields a contradiction as in the above argument that
(3) and (1) cannot happen for different values of $w$. Hence $b$
occurs after the midpoint of $Y$ in the first $h $ and, similarly,
$b^{-1}$ occurs before the mid-point of $Y'$ in the second $h $. It
follows that in order for (1) to occur at least one and therefore both
of $b$ and $b^{-1}$ must cancel contradicting the fact that letters
$b_i^{\pm 1}$ cannot cancel for different words. The last part of the
proof implies the more detailed description of (2).

This completes the proof of Lemma~\ref{cancel2}.
\end{proof}

\begin{lem}\label{productlength2} Let $h\in F(A\cup B)$ be a
  non-exceptional $d_G$-geodesic word and let $w$ be as in the
  conclusion of Lemma~\ref{cancel2}. There exists a non-exceptional
  reduced word $\bar h\in F(A\cup B)$ such that $h=\bar h$ in $G$ such that the following hold:
\begin{enumerate}
\item The conclusion of Lemma~\ref{cancel2} applied to $\bar h$ holds for the same $w$ as before.
\item The word $\bar h$ does not contain a subword of a defining relator of length  more than $\frac{6}{10}N$.
\item For any $n\in\mathbb N$ the freely reduced form of $(\bar hw)^n\bar h$ does not contain a subword of a defining relator of length  more than $\frac{99}{100}N$.
\item If $\bar h$ contains a subword of some defining relator $r\in
  R_\ast$ of length  $N/100$ then in $(\bar hw)^n\bar h$ both the
  first and the last occurrence of $\bar h$ have subwords  of length
  $N/500$ of some element of $R_\ast$ that survive the free cancellation (of $(\bar hw)^n\bar h$).
\end{enumerate}

\end{lem}

\begin{proof} We start with $h$ and then modify this word while
preserving properties (1), (2) above and the fact that the word is
non-exceptional. The final word obtained in this process is $\bar h$.

Choose $u$ maximal such that the freely reduced form of $hw$ has he form $u\hat hu^{-1}$. Since $h$ is non-exceptional and $w$ critical, it follows that there exists $\tilde h$ such that $(hw)^nh=u\hat h^n\tilde h$ (as words).  Then $h=u\hat hv^{-1}$ and therefore $(hw)^nh=u\hat h^{n+1}v^{-1}$ (both as words).

We first show that if (3) does not hold for $h$ then we can find a subword
in $hwh$ that is a subword of a defining relator whose
length is more than $\frac{99}{100}N$. We distinguish the cases the $u$ is critical and that it is not.

If $u$ is critical then $u\hat h^n\tilde h$ differs from $\hat
h^{n+1}$ only by multiplication from left and right by at most 2
critical words. Suppose that $u\hat h^n\tilde h$ contains a subword
$t$ of a defining relator of length $\frac{99}{100}N$ that is not
readable in $u\hat h^n\tilde h$. We assume that $n$ is chosen minimal
for $t$. If $n\ge 3$ then $t$ is periodic, contradicting
genericity. If $t=2$ then $t$ contains a subword of length
$\frac{1}{5}N$ twice, contradicting the small cancellation
hypothesis. It follows that $t\le 1$, as desired.

If $u$ is not critical  then $h=u\hat hv^{-1}$ and therefore $(hw)^nh=u\hat h^{n+1}v^{-1}$ (both as words). Now any subword of $(hw)^nh=u\hat h^{n+1}v^{-1}$ that cannot already be seen as
a subword of $u\hat h^2v^{-1}$ would either contain a significant (in
relation to $N$) 
suffix of $u$ and a significant prefix of $v^{-1}$ or would contain a
significant subword of $\hat h^n$ which is not a subword of $\hat
h^2$. If the former happens then $u$ and $v$ must have a substantial common suffix as they only differ by left multiplication with a critical word, thus the relator must have a long subword that occurs twice, contradiction the 
cancellation properties of the presentation (!!) of $G$, and if the latter happens we argue as in the case of critical $u$. 

Thus we can assume that either (3) holds or that $u\hat h\hat h
u^{-1}$ contains a subword that is a subword of a defining relator $r$
whose length is more than $\frac{99}{100}N$. Genericity and (2) now
imply that $\hat h=W_1W_2W_3$ such that $W_3W_1$ is a subword of $r\in
R_\ast$ of length at least $\frac{98}{100}N$ and that $|W_1|=|W_3|$,
in particular $r=W_3W_1\bar W$ with $\bar W$ of length at most
$\frac{2}{100}N$. We replace $\hat h$ with $\dot{h}=W_1W_2\bar
W^{-1}W_1^{-1}$ which does not change the element of $G$ represented
by the word. After free reduction of $u\dot{h} u^{-1}w^{-1}$ we obtain a new non-exceptional word that by construction satisfies (2). It also satisfies condition (1) as the modification was only applied if the conclusion of (1) of Lemma~\ref{cancel2} did not hold for $w$ and in that case the conclusion of (3) of Lemma~\ref{cancel2} does not hold for $w$ after the modification.

Assertion (4)  follows immediately from the above argument. This is true as (4) holds for the initial $h$ and is preserved under modifications that are made.
\end{proof}

\begin{lem}\label{cancel} Let $h$ be as in the hypothesis of
  Lemma~\ref{cancel2} 
and let $w$ be a
    non-trivial freely reduced critical word.

   Let further $W$ be the word obtained by freely reducing the product $hwh^{-1}$. Write $W$ as a reduced product $W=W_1W_2W_3$, where $W_1$, $W_2$ and $W_3$
    are the portions of  $h$, $w$ and $h^{-1}$ accordingly that survive a  free reduction of
    the product $hwh^{-1}$.

  Then the following hold
  \begin{enumerate}
  \item The free cancellation does not involve a letter $b_i^{\pm 1}$ of $h^{\pm
      1}$ or a
    subword $y$ of $h^{\pm
      1}$ of length $|y|>\lambda_1N/10$ such that $y$ is a subword of some defining relator.

  \item  Let $y$ be a
    subword of $W$ such that $y$ is also a subword of some defining
    relator $r$ that overlaps both $W_1$ and $W_3$ in
    $W$. Then the shorter of these two overlaps has length $\le
    \lambda_1 N/10$ and $|W_2|\le \lambda_1N/1000$.

  \end{enumerate}
\end{lem}

\begin{proof}

(1) The first claim is obvious while the later follows from the
assumptions in Condition~\ref{cond:gen} and Condition~\ref{cond:gen1}.

For part (2), suppose that $y$ overlaps both $W_1$ and $W_3$. Hence
$W_2$ is a subword of $y$ and hence of $r$. Therefore
Condition~\ref{cond:gen} implies that $|W_2|\le \lambda_1|r|/1000$.
Suppose that (2) fails and that the overlaps of $y$ with each of
$W_1$, $W_3$ have lengths $> \lambda_1 |r|/10$. Thus $y=y_1W_2y_3$
where $y_1$ is a terminal segment of $W_1$, where $y_3$ is an initial
segment of $W_3$ and where $|y_1|,|y_3|> \lambda_1 |r|/10$.  Recall
that $W=W_1W_2W_3$ is the freely reduced form of $hwh^{-1}$. Since $|W_2|\le
\lambda_1|r|/1000$, the definitions of $W_1,W_2,W_3$ then imply that
there is a subword $y_0$ of $y_1$ of length $|y_0|\ge \lambda_1 |r|/30$
such that $y_0^{-1}$ is also a subword of $y_3$. Hence $r$ contains
two disjoint occurrences of $y_0$ and $y_0^{-1}$, where $|y_0|\ge
\lambda_1 |r|/30$. This contradicts the genericity assumptions in
Condition~\ref{cond:gen}.
\end{proof}

We now have all the tools necessary to study products of type
$$w_0h^{\varepsilon_1}w_1h^{\varepsilon_2}w_2\cdot\ldots\cdot
w_{k-1}h^{\varepsilon_k}w_k$$ where $h$ is non-exceptional and all
$w_i$ are critical. Possibly after replacing $h$ by another
non-exceptional word as in  Lemma~\ref{productlength2}, we may assume
that the conclusions of Lemma~\ref{productlength2} and
Lemma~\ref{cancel2} and  hold with $w$ provided by Lemma~\ref{cancel}. Thus the following lemma covers all cases.

\begin{lem}\label{lem:casea} Let $h$ be non-exceptional word and $w$ be a critical word such the conclusions of Lemma~\ref{cancel2}, Lemma~\ref{productlength2} and Lemma~\ref{cancel} hold for $h$ and $w$.

Let $p=w_0h^{\varepsilon_1}w_1h^{\varepsilon_2}w_2\cdot\ldots\cdot
w_{k-1}h^{\varepsilon_k}w_k$ where all $w_i\in F(A)$ are
critical. Suppose that for some $b_i\in B$ the word $p$ represents an element of $G$ conjugate to
$b_i$ in $G$.  Then $p$ is conjugate to $b_i$ in $F(A\cup B)$.
\end{lem}

\begin{proof}

\smallskip If $p$ cyclically reduces to $b_i$ in $F(A\cup B)$, there
is nothing to show. Thus we can assume that this does not happen. We
denote the word obtained from the freely reduced form of $p$ by cyclic
reduction by $\tilde p$. It now follows from  Lemma~\ref{lem:sc-conj}
(3) that  $\tilde p\tilde p$ contains a subword that is
$\frac{9999}{10000}$ of some relation from $R\ast$.

\smallskip We will show that this is not possible. We distinguish two
cases, namely the case that $h$ contains a subword of length
$\frac{1}{100}N$ that is also a subword of a defining relation and the
case that it does not contains such a subword but contains some letter
from $B^{\pm 1}$. Since $h$ is assumed to be non-exceptional, these
two cases cover all the possibilities.

\smallskip\noindent {\bf Case 1:} Suppose first that $h$ contains a
subword of length $\frac{1}{100}N$ that is also a subword of a
defining relation from $R_\ast$. Recall that $w$ is chosen as in
Lemma~\ref{cancel2}. In the following we will denote the word obtained
from $(hw)^{n-1}h$ by free reduction by $h_n$.

\medskip We may assume that $p$ is of the form $$p=h_{j_1}^{\varepsilon_1}w_1\cdot\ldots\cdot w_{k-1}h_{j_k}^{\varepsilon_k}w_k$$  where $\varepsilon_i=\pm 1$ for $1\le i\le k$  and that $w_i$ are critical words such that  $w_i\neq w^{\varepsilon_i}$ if $\varepsilon_i=\varepsilon_{i+1}$, otherwise we could replace the subword $h_{j_i}^{\varepsilon_i}w_ih_{j_{i+1}}^{\varepsilon_{i+1}}=h_{j_i}^{\varepsilon_i}w^{\varepsilon_i}h_{j_{i+1}}^{\varepsilon_i}$ by $h_{j_i+j_{i+1}}$. We may further assume that this product is freely reduced, i.e. that $w_i\neq 1$ if $\varepsilon_i=-\varepsilon_{i+1}$, and that it is cyclically reduced, i.e. that either $\varepsilon_1=\varepsilon_k$ or that $\varepsilon_1=-\varepsilon_k$ and $w_k\neq 1$. 


\medskip
By Lemma~\ref{productlength2} holds, the word $h_n$ does not contain a
subword which is more than $\frac{99}{100}$ of a defining relation. By
part (4) of Lemma~\ref{productlength2}, any  $h_n$ contains two
subwords of relations of length at least $|r|/500$ inherited from the
first and last occurrence of $h$ by the free cancellation of the
product $(hw)^{n-1}h$. In particular we can choose words $U$ and $V$
that are the leftmost, respectively rightmost, subwords of $h_n$ that
represent $\frac{1}{500}$ of a defining relation. Then these words are the same for all $h_n$ with $n\ge 1$.

\medskip
It now follows from Lemma~\ref{cancel2} and Lemma~\ref{cancel} that
none of the occurrences of $U^{\pm 1}$ and $V^{\pm 1}$ (which are subwords of the $h_{j_i}^{\varepsilon_i}$) cancel
completely in the free and cyclic reduction of $p^2$ to $\tilde
p^2$. 

It also follows that the remnants of the $U^{\pm 1}$ and $V^{\pm
  1}$ coming from two adjacent $h_{j_i}^{\varepsilon_i}$ and
$h_{j_{i+1}}^{\varepsilon_{i+1}}$ do not lie in a subword of $\tilde
p^2$ that is also a subword of a defining relation.

Since the words $w_i$ do not  contain subwords that are greater than
$\frac{1}{10000}$ portions of defining relations,  it follows that any
subword of a defining relation that occurs in $\tilde p\tilde p$ is of
length at most $\frac{99}{100}+2\frac{1}{500}+2\frac{1}{10000}$ of
that relation. Since this number is smaller than $\frac{9999}{10000}$,
it follows that $p$ cannot represent an element of $G$ conjugate to
$b_i$ in $G$.

\medskip
\noindent {\bf Case 2:} Assume now that $h$ contains a letter of
$B^{\pm 1}$ but does not contain a subword of length $\frac{1}{100}N$
that is also a subword of a defining relation.

Let $u$ be the maximal prefix of $h$ that is a word in $F(A)$ and let $v^ {-1}$ be the maximal suffix of $h$ that is a word in $F(A)$. Thus either $h=ub_i^{\varepsilon}\tilde
wb_j^{\eta}v^{-1}$ or $h=ub_i^ {\varepsilon}v^ {-1}$. 
Let $\hat w$ be the unique reduced word in $F(A)$ such that $v^{-1}\hat wu$ reduces to the trivial word.

\smallskip Now let  $Z_n$ obtained by the word obtained from $(h\hat w)^nh$ by free reduction. Then $Z_n$ does not contain a subword of length $\frac{1}{10}N$ that is also a subword of a defining relator as such a word would need to be periodic
because of the assumption that $h$ does not contain more than
$\frac{1}{100}$ of a relation.

\smallskip Moreover, the word $Z_n$ has a prefix
$ub_i^{\varepsilon}$ and a suffix $b_j^{\eta}v^{-1}$ and hence $Z_n$
contains a letter from $B^{\pm 1}$. We can assume that
$p$ is of the form $$p=Z_{j_1}^{\varepsilon_1}w_1\cdot\ldots\cdot
w_{k-1}Z_{j_k}^{\varepsilon_k}w_k$$  where $\varepsilon_i=\pm 1$ for
$1\le i\le k$  and the $w_i$ are reduced critical words such that $w_i\neq \hat w^{\varepsilon_i}$ if
$\varepsilon_i=\varepsilon_{i+1}$; otherwise we could replace the
subword
$Z_{j_i}^{\varepsilon_i}w_iZ_{j_{i+1}}^{\varepsilon_{i+1}}=Z_{j_i}^{\varepsilon_i}\hat
w^{\varepsilon_i}Z_{j_{i+1}}^{\varepsilon_i}$ by
$Z_{j_i+j_{i+1}}$.

\smallskip Therefore at least one letter from $B^{\pm 1}$ of
each $Z_{j_i}^{\varepsilon_i}$ and one letter from $A^{\pm 1}$ of each
$w_i$ do not cancel in the free reduction of $p$. Thus, since each of
these factors does not contain more than $\frac{1}{10}$ of a defining
relation, the freely reduced word obtained from $p^2$ cannot contain
$\frac{1}{2}$ of a defining relation. This contradicts the fact that
$p^2$ is conjugate to $b_i^2$ in $G$, thus ruling out Case~2. 
\end{proof}

\begin{lem}\label{ztobarz} Let $z$ be a reduced word in $F(A\cup B)$
  and  suppose that $\bar z\in F(A)$  is an exceptional word
  representing the same element in $G$ as $z$. Then the following hold:
  \begin{enumerate}
\item There exist words $\alpha_0,\ldots,\alpha_t,\beta_1,\ldots,\beta_t$ (nontrivial all except possibly $\alpha_0$ and $\alpha_t$) and words $\bar\beta_1,\ldots,\bar\beta_t$ such that the following hold:

\begin{enumerate}
\item We have $z=\alpha_0\beta_1\alpha_1\ldots \beta_t\alpha_t$
  (equal as words).
\item We have $\bar z=\alpha_0\bar\beta_1\alpha_1\ldots
  \bar\beta_t\alpha_t$ (equal as words).
\item We have $|\bar\beta_i|\le \frac{1}{100}|\beta_i|$ and $|\beta_i|\ge 1000N$ for all $i$, where $N$ is the length of the defining relations of $G$.
\end{enumerate}

\item If $u$ and $\bar u$ are maximal words such that $z=uwu^{-1}$ and
  $\bar z=\bar u\bar w{\bar u}^{-1}$ (equal as words) for some words $w$ and $\bar w$ then $|\bar w|\le |w|$.

\item If $a_i^{
\eta}$ is a letter of $u$ such that neither this occurrence of
$a_i^{\eta}$ nor corresponding the
corresponding occurrence of $a_i^{-\eta}$ in $u^{-1}$ fall in subwords
$\beta_j$ of $z$ 
then in $\bar z$ the letter $a_i^{\eta}$ occurs in $\bar u$ and the
letter $a_i^{-\eta}$ occurs in $\bar u^{-1}$.
\end{enumerate}
\end{lem}

\begin{proof}

Recall that the presentation (!!) for $G$ satisfies the $C'(\lambda_0)$
condition with $\lambda_0=1/1000$.

By a \emph{$\lambda_0$-reduction} on a word in the generators of $G$
we mean replacing a subword of that word, that is also a subword of
one of the defining relations $r$ of length $> (1-3\lambda_0)|r|$, by
the complementary portion of $r$. Recall that, by
Definition~\ref{defn:l-reduced}, a freely reduced word in $F(A\cup B)$
is \emph{$\lambda_0$-reduced} if it does not admit any $\lambda_0$-reductions.

Note that by definition of being exceptional, the word $\bar z\in
F(A)$ is $\lambda_0$-reduced. Observe now that if $\tilde z$ is any other freely reduced and $\lambda_0$-reduced representative of $z$ then $\tilde z=\bar z$ as
words. Indeed, if $\tilde z\ne \bar z$ then the
equality diagram for $\tilde z=_G \bar z$ has the form provided by
Proposition~\ref{prop:eq-conj} and, also by Proposition~\ref{prop:eq-conj}, the word $\bar z$ contains a subword of a
defining relation of length $\ge 1/100$ of that relation. This is
impossible since $\bar z$ is $\lambda_0$-reduced.

Thus any maximal sequence of free and $\lambda_0$-reductions applied to $z$ must
terminate with the word $\bar z$, and $\bar z$ is the unique freely
reduced and $\lambda_0$-reduced word representing the same element as $z$ in $G$.

Let $\bar z=_G \bar u\bar w{\bar u}^{-1}$ be the representation of the
element $z\in G$ provided by Lemma~\ref{lem:sc-conj}. Recall that by Lemma~\ref{lem:sc-conj}, the element of $G$ represented by the word $\bar
w$ is the shortest (in $G$) representative of the conjugacy class of
$\bar z$, the word  $\bar u\bar w{\bar u}^{-1}$ is freely reduced and
Dehn reduced (in particular, $\lambda_0$-reduced). Since both
$\bar z$ and $\bar u\bar w{\bar u}^{-1}$ are freely reduced and $\lambda_0$-reduced, it follows
that $\bar z=\bar u\bar w{\bar u}^{-1}$ as
words. Moreover, we can obtain $\bar z$ from $z$ by applying
$\lambda_0$-reductions and free cancellations. Note that since the element represented by the geodesic word $\bar w$ is shortest in the $G$-conjugacy class of $z$, part (2) of Lemma~\ref{ztobarz} is now established.

Every $\lambda_0$-reduction replaces a subword $Q$ by a word of length
$\le |Q|/100$. A free reduction replaces a subword of positive even
length by an empty word. Recall that $\bar z$ is the unique freely
reduced and $\lambda_0$-reduced representative of $z$. Moreover, since
$z\ne \bar z$, at least one $\lambda_0$-reduction is necessary to get
to $\bar z$ from $z$. Hence, by an inductive argument, we can show that $z$ and $\bar z$ can be represented as
$z=\alpha_0\beta_1\alpha_1\ldots \beta_t\alpha_t$, $\bar z=\alpha_0\bar\beta_1\alpha_1\ldots \bar\beta_t\alpha_t$, where
$|\bar\beta_i|\le \frac{1}{100}|\beta_i|$, where
$\beta_i=_G\bar\beta_i$ and where the process of $\lambda_0$-reducing $z$ to $\bar z$ consists in $\lambda_0$-reducing each of $\beta_i$ to $\bar \beta_i$ separately. This implies part (3) of Lemma~\ref{ztobarz}.

The fact that $|\beta_i|\ge 1000N$ requires a more involved argument
which we sketch below, omitting some of the technicalities.

The process of $\lambda_0$-reducing $z$ to $\bar z$ can be conducted
in stages. At the first stage we find a maximal collection of
non-overlapping subwords $y_q$, $q=1,2,...,M$, of $z$ such that each
of them is also a subword of some defining relator $b_j=v_j(\overline
a)$ (or a cyclic permutation of this relator or its inverse) of length
comprising $\ge (1-3\lambda_0)$ fraction of that relator. Since $z$ is
a word in $F(A)$, each $y_q$ is also a word in $F(A)$.  We then replace
each of these subwords $y_q$ by the missing portions of the
corresponding defining relators, which introduces $M$ letters from
$\{b_1,\dots, b_m\}^{\pm 1}$ into the word, and denote the result by
$z_0'$. We then freely reduce $z_0'$ to get a word $z_1$ which is the
result of the first stage. Note that no $b_j^{\pm 1}$ from $z_0'$ get
cancelled in freely reducing $z_0'$ to $z_1$ since otherwise one could
show that the word $z$ was not freely reduced. Also, we have $t\le M$
and each of $y_q$ is a subword of some $\beta_i$ from the conclusion
of the lemma. Recall that the $\lambda_0$-reduced form $\bar z$ of $z$
does not involve any $b_j^\pm 1$.

In the word $z_1$ obtained after the first stage the only applicable
$\lambda_0$-reductions are those that involve large portions of the
defining relators $a_j=u_j(\overline b)$ (with some technical
exceptions which we suppress for the purposes of this sketch).

At the second stage, we apply a similar process to $z_1$ by
finding a maximal collection of disjoint subwords $x_s$ of $z_1$ such
that each $x_s$ is also a subword of some relation $a_j=b_j(\overline
b)$ comprising $\ge (1-3\lambda_0)$ fraction of that relation. Thus
each $x_s$ is, except for a single letter, a word in
$\{b_1,\dots,b_m\}^\pm 1$ of length $\ge 0.99N$. When tracked back to
$z$, this causes, for each $x_s$, a conglomeration of $\ge 0.99N$ of
the original $y_q$s into a single $\beta_j$ of length $\ge
(1-3\lambda_0)N\cdot 0.99N\ge 1000N$ for sufficiently large $N$. Some
of the $b_j$ introduced in $z_1$ at stage one may not be contained in
any of the $x_s$ at stage two.

However, eventually all of the  $b_j$
that were introduced in $z_1$ at stage one have to disappear in the
$\lambda_0$-reduction process (since the word $z$ involves only
$\{a_1,\dots, a_k\}^{\pm 1}$) and one can show that the condition $|\beta_i|\ge 1000N$ holds for every $i=1,\dots, t$.

\end{proof}

\section{Proof of the main theorem}

This section is dedicated to the proof of the main result of the paper,
Theorem~\ref{main} from the introduction. While the strategy of the proof is similar to the strategy of the proof of
Theorem~\ref{specialcase}, the argument turns out to be significantly more involved.

\smallskip We will prove the following theorem which is the precise formulation of Theorem~\ref{main}
from the Introduction.

\begin{thm}\label{thm:main} Let $k\ge 2$,  $m\ge 2$ and let
  $A=\{a_1,\dots, a_k\}$ and $B=\{b_1,\dots, b_m\}$ be free bases of
  $F(A)$ and $F(B)$. Then there exist exponentially generic subsets $V\subseteq
  \mathcal C_{m,A}$ and $U\subseteq \mathcal C_{k,B}$ with the following properties.

Let $(u_1(\bar b),\dots,u_k(\bar b))\in U$ and $(v_1(\bar a), \dots, v_m(\bar a))\in V$
be arbitrary such that $|u_i|=|v_j|$ for $1\le i\le k$ and $1\le j\le m$.
Let $G$ be given by the presentation
\[
G=\langle a_1,\ldots,a_k,b_1,\ldots b_m\,|\,a_i=u_i(\bar b),
b_j=v_j(\bar a)\hbox{ for }1\le i\le k, 1\le j\le m\rangle.\tag{!!}
\]

Then the $(k+1)$-tuple
 $(a_1,\ldots, a_k,1)$ is not Nielsen equivalent in $G$ to a $(k+1)$-tuple of form $(b_1,b_2,*,\ldots,*)$.
\end{thm}

\medskip The proof of Theorem~\ref{thm:main} is by
contradiction. Thus we assume that the $(k+1)$-tuple $(a_1,\ldots,
a_k,1)$ is Nielsen equivalent in $G$ to a tuple of type
$(b_1,b_2,*,\ldots,*)$. It follows that in $F(A)$ the tuple
$(a_1,\ldots,a_k,1)$ is Nielsen equivalent to a tuple of type
$(w_1(\bar a),\ldots ,w_k(\bar a),w_{k+1}(\bar a))$ such that $b_i=w_i(\bar a)$ in $G$ for $i=1,2$.

\medskip Let $T=(w_1(\bar a),\ldots ,w_k(\bar a),w_{k+1}(\bar a))$ be
a tuple with the above properties such that
$|w_1|_A+|w_2|_A$ is minimal among all such tuples.

\medskip Let $S_T$ be the $A$-graph as in Section~\ref{stallings} and choose a sequence $$S_T=\Gamma_0,\Gamma_1,\ldots
,\Gamma_l=R_A$$ of $A$-graphs such that $\Gamma_{i}$ is obtained from
$\Gamma_{i-1}$ by a Stallings fold for $1\le i\le l$. Choose $q$
maximal such that $\Gamma_q$ does not contain a subgraph isomorphic to $R_A$ as an $A$-graph. Put $\Delta:=
\Gamma_q$. Then $\Delta$ contains a subgraph $\Psi$ of rank $k$ with
at most $k+2$ edges such that $\Psi$ is a core graph 
and that $\Psi$, via a single fold, folds onto a
subgraph of $\Gamma_{q+1}$ that is isomorphic to either $R_A$ or to
$R_A$ with a spike-edge, see the proof of Lemma~\ref{trivial}. Note that any word readable in $\Psi$ is critical.

Note that $Core(\Delta)$ cannot coincide with $\Psi$ by the same
argument used in the proof of Theorem~\ref{specialcase}. 
Indeed, if $Core(\Delta)=\Psi$, then by Lemma~\ref{trivial} there
exists a reduced word of length $4$ that cannot be read in
$Core(\Delta)$. On the other hand, the cyclically reduced form of
$w_1(\overline{a})$ can be read in
$Core(\Delta)=Core(\Gamma_q)$. Since $w_1(\overline a)$ is conjugate
to $b_1$ in $G$, it follows that the cyclically reduced form of
$w_1(\overline a)$
must contain a large portion of some defining relation $v_j(\overline
a)b_j^{-1}$. Hence, by the genericity assumptions on $v_j$, the cyclically
reduced form of $w_1(\overline a)$ contains every freely reduced word
of length $4$ as a subword, yielding a contradiction. Thus indeed
$Core(\Delta)\ne \Psi$.

\medskip In the following a \emph{lollipop} is a loop (candy) with an segment (the stick) attached. We will allow the stick to be degenerate, i.e. the lollipop to be a loop.

\begin{lem}\label{coredescription} The $A$-graph $\Delta$ has the following properties:

\begin{enumerate}
\item The subgraph $\bar \Delta=Core(\Delta)$ carries the
  fundamental group of $\Delta$ and one of the following holds (see Figure~\ref{fig:delta}):
  \begin{enumerate}
  \item[(a)] $\bar \Delta$ can be obtained from $\Psi$ by attaching a
    (subdivided) segment $\gamma$ along both endpoints to distinct vertices of $\Psi$.
  \item[(b)] $\bar \Delta$ can be obtained from $\Psi$ by attaching a
    (subdivided) lollipop along its unique valence one vertex if the stick is non-degenerate and an arbitrary vertex otherwise.
  \end{enumerate}
\item Either the base-vertex of $\Delta$ lies in $\bar \Delta$ in
  which case $\Delta=\bar \Delta$, or $\Delta$ is obtained from $\bar
  \Delta$ by attaching a segment to $\bar \Delta$ along one of its
  endpoints.  The other endpoint of the segment is then the base
  vertex $x_0$ of $\Delta$.
\end{enumerate}
\end{lem}

	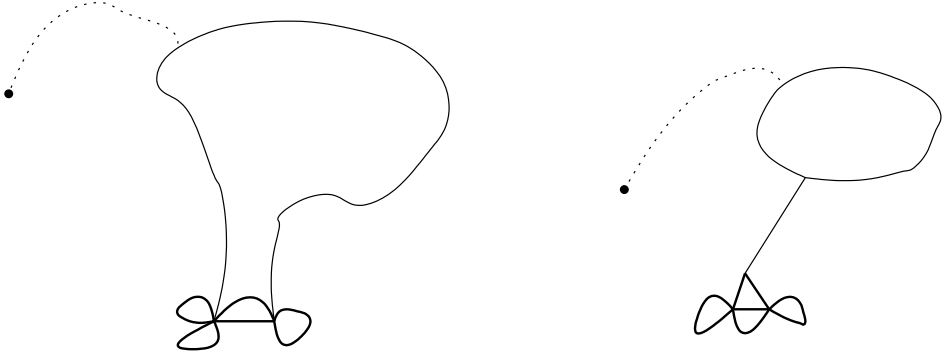
\begin{figure}[htb]
		\input{delta.pstex_t}
		\caption{The two options for the graph $\Delta$}
		\label{fig:delta}
	\end{figure}

\begin{proof} Part (1) follows immediately from the fact that $\bar\Delta$ is a core graph of rank $k+1$ and that $\Psi$ is a core subgraph of rank $n$. Part (2) follows as any vertex of $S_T=\Gamma_0$, except possibly the base vertex, is the initial vertex of two edges with distinct labels. This property is preserved by folds and therefore holds for $\Delta$. The claim is now immediate.
\end{proof}

Note that in Figure~\ref{fig:delta} the graph $\Psi$ consists of the fat lines. A trivial but important fact is that any of the words $w_i(\bar a)$  can be read by a path in $\Delta$ based at the base vertex since the map $p$ from $S_T$ to $R_A$ factors through $\Delta$. This path is clearly reduced as the words $w_i(\bar a)$ are reduced.

\begin{conv}[Pivotal path and pivotal word]
\medskip Let the word $z\in F(A)$, called the \emph{pivotal word}, be defined as follows: In the case
(1)(a) of Lemma~\ref{coredescription} put $z$ to be the label of the segment attached to $\Psi$. In the case (1)(b) of
Lemma~\ref{coredescription} put $z$ to be the label of the closed
reduced edge-path in $\bar \Delta$ based at the attaching point of the
lollipop that walks once around the lollipop, i.e. it is a word of
type $uwu^{-1}$ where $u$ is the label of the stick of the lollipop.

Thus in both cases $z$ is the label of a path $Z$ in $\Delta$ which we
call the \emph{pivotal path}.
\end{conv}

\medskip Let now $\bar z$ be a geodesic word in $G$ over $(A\cup
B)^{\pm 1}$ such that in $G$ we have the
identity $z=_G \bar z$. We will distinguish two cases, namely that $\bar z$ is exceptional and that $\bar z$ is non-exceptional.

\medskip\noindent {\bf Case A:} The word $\bar z\in F(A\cup B)$ is non-exceptional.

\medskip Possibly after replacing $\bar z$ with another non-exceptional word representing the same element we can assume that for some unique critical word $w$ the conclusions of Lemma Lemma~\ref{cancel2}, Lemma~\ref{productlength2} and Lemma~\ref{cancel} hold for $\bar z$ and $h$.

Let now $x_1$ be a base point in $\Psi$ that is incident to the segment or lollipop with label $z$. We study elements
represented by closed reduced paths in $\Delta$ based at $x_1$.  It
follows from the construction that for $i=1,2$ there exist closed
reduced paths $\alpha_1,\alpha_2$ in $\Delta$ based at $x_1$ with
labels $p_1, p_2$ such that for
$i=1,2$ the word $p_i$ represents an element of $G$ conjugate to
$b_i$ in $G$. The description of $\Delta$ in Lemma~\ref{coredescription} and
the definition of $z$ imply that $p_1$ and $p_2$  are words that are freely equivalent to words of type
$$w_0z^{\varepsilon_1}w_1\cdot\ldots\cdot
w_{n-1}z^{\varepsilon_n}w_n$$ where the $w_i$ are reduced words read
by reduced paths in $\Psi$ and $\varepsilon_i\in\{-1,1\}$ for all
$i$. In particular the words $w_i$ are critical. As in $G$ we have $z=\bar z$, it follows that there are also words $\bar p_1$ and $\bar p_2$ of type $$w_0{\bar z}^{\varepsilon_1}w_1\cdot\ldots\cdot w_{n-1}{\bar z}^{\varepsilon_n}w_n$$ with the $w_i$ as before that represent elements conjugate to $b_1$ and $b_2$. Note that the obtained words might no longer be freely reduced even if the initial ones were.

\medskip Suppose now that $\bar p_1=w_0{\bar
  z}^{\varepsilon_1}w_1\cdot\ldots\cdot w_{n-1}{\bar
  z}^{\varepsilon_n}w_n$. It now follows from Lemma~\ref{lem:casea}
that $\bar p_1$  cyclically reduces to $b_1$ in $F(A\cup B)$.  The proof of Case~2 of
Lemma~\ref{lem:casea} implies that of each factor $\bar z^{\pm 1}$ at
least one letter $b_i^{\pm 1}$ does not cancel. This implies that
$w_1\bar z^{\pm 1}\bar w_1$ must cyclically  reduce to $b_1$ where
$w_1$ and $\bar w_1$ are words in $F(A)$. Therefore
$\bar z$ contains a single letter of type $b_i^{\pm 1}$, namely
$b_1^{\pm 1}$. The same argument applied to $p_2$ implies that $\bar
z$ contains $b_2^{\pm 1}$ which yields a contradiction.   This
concludes the proof of Theorem~\ref{thm:main} for Case A.

\medskip\noindent {\bf Case B:} The word $\bar z$ is exceptional (so
that, in particular, $\bar z\in F(A)$). We
distinguish the cases that $z=\bar z$ and that $z\neq\bar z$ (as words).

\smallskip \noindent{\bf Subcase B1: } Suppose that  $z\neq\bar z$.
The aim in this case is to get a contradiction to the minimality
assumption on the sum of the lengths of $w_1$ and $w_2$. The
underlying idea is simple: we just replace in $\Delta$ the segment
(respectively lollipop) with label $z$ by an segment (respectively
lollipop) with label $\bar z$ and then carry out the corresponding
replacement in the loops of $S_T$ upstairs. Since in this case the
word $\bar z$ is considerably shorter than $z$, this yields a
contradiction with the minimality assumption on the pair $(w_1,w_2)$. However taking the basepoint into account makes the proof more technical.

\medskip Since Lemma~\ref{ztobarz} applies to the transition from $z$
to $\bar z$, we choose the words $\beta_i$ and $\alpha_i$ as in the conclusion of Lemma~\ref{ztobarz}.
 Part (1c) of Lemma~\ref{ztobarz} implies that $\bar z$ is significantly shorter than $z$, namely by at least $\frac{99}{100}\sum|\beta_i|\ge 990\cdot k\cdot |v_i|$. We denote $\sum|\beta_i|$ by $L$.

\medskip Any vertex $x$ of the pivotal path $Z$, except the initial and terminal
vertex of $Z$, corresponds to a subword of length
two of $z$ which is read along the two-edge subpath $e_1e_2$ of $Z$ when reading
$z$, where $x=t(e_1)=o(e_2)$. In the case of the lollipop and the
vertex lying on the stick there are two such two-letter subwords of $z$
associated with the vertex $x$. 

We say that a vertex on the pivotal path $Z$ is \emph{inessential}
if none of these at most two subwords is a subword of one the
$\beta_i$. Note that, by definition, the initial and the terminal
vertices of $Z$ are inessential.

\medskip We perform some changes to $S_T$ and $\Delta$; these changes do not alter the following facts:
\begin{enumerate}
\item $S_T$ folds onto $\Delta$ and  $\Delta$ folds onto $R_A$. Thus $S_T$ folds onto $R_A$.
\item The first two loops of $S_T$ are labeled by words in the $a_i^{\pm 1}$ that represent $b_1$ and~$b_2$.
\end{enumerate}

The conclusion then follows from the fact that after the final modification the length of the first two loops of $S_T$ has decreased.

	\begin{figure}[htb]
		\input{delta2.pstex_t}
		\caption{Fat parts correspond to paths read by the $\beta_i$}
		\label{fig:delta2}
	\end{figure}
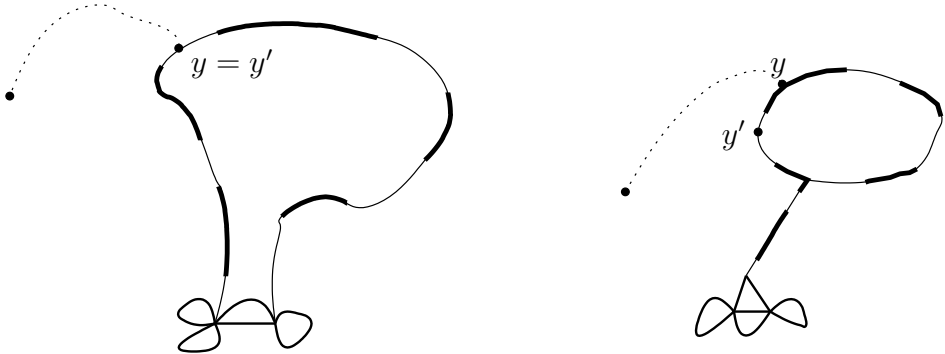

\smallskip

Let $y$ be the vertex of $\bar\Delta=Core(\Delta)$ such that $d(v_0,y)=d(v_0,\bar\Delta)$. Thus $y=v_0$ if $v_0\in\bar \Delta$ and $y$ is the attaching point of the segment joining $\bar\Delta$ and $v_0$ otherwise.  We now choose a new vertex $y'\in\Delta$ so that the following hold (see Figure~\ref{fig:delta2}):
\begin{enumerate}
\item The vertex $y'$ is inessential.
\item We have $d(y,y')\le 2L$.
\item If $y$ is in distance less than $L$ from the stick of the lollipop then $y'$ is on the stick of the lollipop (if we are in the lollipop situation).
\end{enumerate}

The existence of such a vertex $y'$ follows immediately from the
definition of $L$. From now on we only provide figures for the
lollipop situation as it is the more subtle one.

\smallskip Let $[y',y]$ be an arbitrary geodesic  in $\bar\Delta$. We now
unfold along $[y',y]$ (starting at $y$). In the end $v_0$ is attached to
$\bar\Delta$ by a segment of length $d(v_0,y)+d(y,y')$  which intersects $\bar\Delta$ only in $y'$.  We denote the just constructed unfolded $\Delta$ by $\Delta_+$. Note that we have not changed $\Psi$ or $Core(\Delta)$ since  $Core(\Delta_+)=Core(\Delta)=\bar\Delta$.

	\begin{figure}[htb]
		\input{deltaunfold.pstex_t}
		\caption{Unfolding $\Delta$ to $\Delta_+$}
	\end{figure}
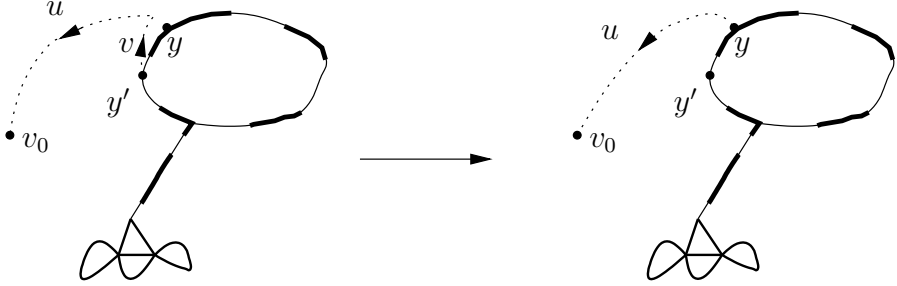

\smallskip Now $S_T$ does not necessarily map onto $\Delta_+$  but it does if we
introduce at the beginning and at the end of each petal of $S_T$ the trivial
word that corresponds to the unfolding when going from $\Delta$ to $\Delta_+$. More precisely:

\smallskip If $u$ is the word corresponding to $[y,v_0]$ in $\Delta$ and $v$
the word corresponding to $[y',y]$ in  $\Delta$ then the $i$-th
petal in $S_T$ has a label of type $u^{-1}d_iu$ for some word
$d_i$. This happens because the $i$-th petal in $S_T$ gets
mapped onto a path in $\Delta$ that starts and ends at $v_0$, after it starts
(before it ends) at $v_0$ it has no other option than going to $y$ first
(or must come from $y$).

	\begin{figure}[htb]
		\input{loopofs.pstex_t}
		\caption{A loop of $S_T$}
	\end{figure}
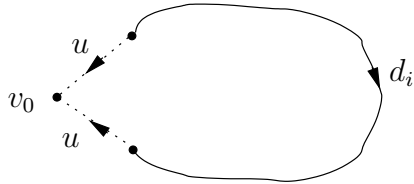

 We replace the $i$-th petal of $S_T$ with a petal with
label $u^{-1}v^{-1}vd_iv^{-1}vu$. These new petals represent the same
group elements as the old ones and now map to  $\Delta_+$ since the word $u^{-1}v^{-1}$ maps to
the edge path $[v_0,y']$ in $\Delta_+$ and the word $u^{-1}v^{-1}v$ onto
$[v_0,y']\cup [y',y]$ in $\Delta_+$. We further modify the newly constructed graph by
performing the free reduction on the subword $vd_iv^{-1}$ that
corresponds to backtracking on the segment, respectively lollipop. Thus we
are left with petal labels of type $u^{-1}v^{-1}p_i\bar d_iq_i^{-1}vu$ where $p_i$ is the prefix of $v$ and $q_i^{-1}$ is the suffix of $v^{-1}$ that have not been cancelled in the free reduction of $vd_iv^{-1}$. We call the new graph $S_T^+$.

	\begin{figure}[htb]
		\input{gammaunfold.pstex_t}
		\caption{Constructing $S_T^++$ (only one loop drawn)}
	\end{figure}
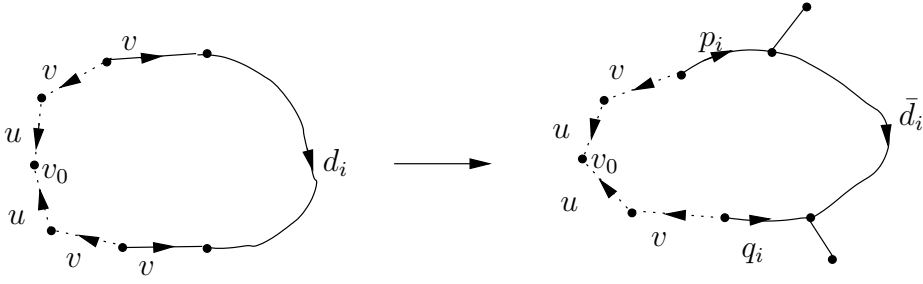

Note that for $i=1,2$ the $i$-th loop of $S_T^+$ still represents the element
$b_i$ in $G$. By construction, since $d(y,y')=|v|\le 2L$, we have increased the length of each petal by at most $8L$.

\smallskip Now the subword $p_i\bar d_iq_i^{-1}$ of the $i$-petal label
in  $S_T^+$ corresponds to a closed
path in $\bar\Delta$ based at $y'$ and is therefore of type
\[
\theta_1\gamma_0uw^{r_1}u^{-1}\gamma_1\cdots
uw^{r_l}u^{-1}\gamma_l\theta_2^{-1} \tag{$\spadesuit$}
\] 
where:

\begin{enumerate}
\item The words $\theta_1$ and $\theta_2$ correspond to reduced paths from $y'$ to $\Psi$. In the case of a lollipop these paths might go around the lollipop a number of times.
\item The words $\gamma_i$ are read in $\Psi$ and are non-trivial for $1\le i\le l-1$.
\item $r_i\in\mathbb Z-\{0\}$ and $r_i\in\{-1,1\}$
  if $z$ is the label of an segment rather than the label of a lollipop.
\end{enumerate}

We now show that $l>20$ for $i=1,2$. Indeed, after conjugating the word in
$(\spadesuit)$ by $\theta_2^{-1}$, we obtain the word
$$\theta_2^{-1}\theta_1\gamma_0uw^{r_1}u^{-1}\gamma_1\cdots
uw^{r_l}u^{-1}\gamma_l=uw^{r_0}u^{-1}\gamma_0uw^{r_1}u^{-1}\gamma_1\cdots
uw^{r_l}u^{-1}\gamma_l.$$ Replacing $z^{r_i}=uw^{r_i}u^{-1}$ by ${\bar
  z}^{r_i}=\bar u\bar w^{r_i}\bar u^{-1}$ and freely reducing the
result produces a word in $F(A)$ that still represents a conjugate of
$b_1$ (or $b_2$) and must therefore contain half a relation after free
reduction. Now this word consists of the remnants of the $\gamma_i$ and the $\bar u\bar w^{r_i}\bar u^{-1}$. Now the $\gamma_i$ are critical and can therefore not contain a significant subword of a defining relation. Moreover the $\bar u\bar w^{r_i}\bar u^{-1}$ do not contains more
than $\frac{1}{50}$-th  of a defining relator by the 
assumption on $\bar z=\bar u\bar w\bar u^{-1}$ and the fact that a subword of a defining relator cannot be periodic.
It follows that $l>20$.

\smallskip We now construct a new $A$-graph $\Delta_*$ from
$\Delta_+$. We first replace $\bar\Delta=Core(\Delta_+)$ by replacing
the segment, respectively lollipop, with label $z$ as follows. In the case
of an segment, we replace it with an segment with label $\bar z$ and the same
attaching points as before. In the case of a lollipop we replace the lollipop with a lollipop whose stick is labeled by $\bar u$ and whose loop based at the stick is labeled by $\bar w$. In both cases walking along the segment or around the lollipop reads $\bar z=\bar u\bar w\bar u^{-1}$. Let $\bar\Delta_*$ be the modified $\bar\Delta$.

	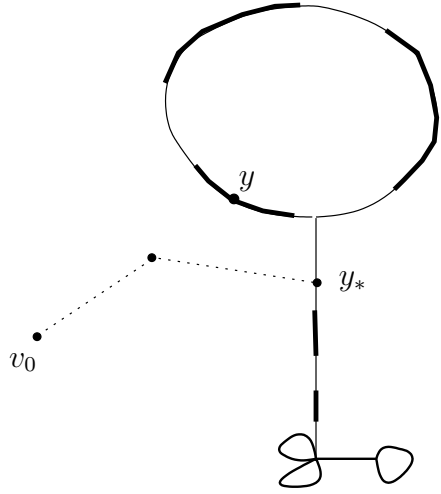
\begin{figure}[htb]
		\input{gammastern.pstex_t}
		\caption{In  $\Delta_*$ the point $y'=y_*$ could move to stick}
	\end{figure}

 Since the vertex $y'$ was chosen to be inessential, it corresponds to
 a unique vertex of $\bar\Delta_*$ which we call $y_*$. We now define
 $\Delta_*$ to be the $A$-graph obtained from $\bar\Delta_*$ by attaching
 the segment $[v_0,y']$ (from $\Delta$) to $\bar\Delta_*$ by
 identifying $y'$ and $y_*$. Note that in the lollipop case if $y'$
 was on the stick of the original lollipop then by part (3) of Lemma~\ref{ztobarz}
 $y_*$ is on the stick of the new one but the converse is not
 necessarily true. The $A$-graph $\Delta_*$ still folds onto $\mathcal
 R_A$ since $\Delta_*$
 still contains $\Psi$ as a subgraph.

\smallskip Next we need to adjust $S_T^+$ so that it maps onto $\Delta_*$.
Recall that the labels of the petals in $S_T^+$ are of type $u^{-1}v^{-1}p_i\bar d_iq_i^{-1}vu$ where $p_i\bar d_iq_i^{-1}$ is of the form $$\theta_1\gamma_0uw^{r_1}u^{-1}\gamma_1\cdots uw^{r_l}u^{-1}\gamma_l\theta_2^{-1}$$ with $\theta_1$ and $\theta_2$ words read by paths that start at $y'$ and end in $\Psi$. We will not change the initial subword $u^{-1}v^{-1}$ or the terminal subword $vu$, nor will we change the subwords $\gamma_i$. We first replace the subwords $uw^{r_i}u^{-1}$ by $\bar u\bar w^{r_i}\bar u^{-1}$. This reduces the length of the word by at least $20\cdot \frac{99}{100}L$.

\smallskip We will now replace $\theta_1$ and $\theta_2$ by words
$\bar\theta_1$ and $\bar\theta_2$ that correspond  to paths in
$\Delta_*$ whose initial vertex is $y_*$, which have the same terminal
vertices  in $\Psi$ as $\theta_1$ and $\theta_2$ and which are at most
as long as $\theta_1$ and $\theta_2$, respectively. We give the construction for $\theta_1$, as the case for
$\theta_2$ is analogous.

 Since $\theta_1$ is freely reduced, it is a suffix of the word $uw^ru^{-1}$ for some $r\in\mathbb Z$ with $r\in\{-1,1\}$ in the segment case. We choose $r$ such that $|r|$ is minimal, w.l.o.g. we can assume that $r>0$. We distinguish a number of cases:

If $|r|=1$ then $\theta_1$ is a suffix of $z$ and the conclusion follows after replacing all subwords $\beta_i$ of $\theta_1$ by $\beta_i'$, our choice of $y'$ guarantees that that each $\beta_i$ either is a subword of $\theta_1$ or does not overlap with $\theta_1$. The resulting word $\bar\theta_1$ is clearly at most as long as $\theta_1$. This case deals in particular with the non-lollipop case.

\smallskip If $|r|\ge 2$ and $\theta_1=tw^ru^{-1}$, i.e. if $y'$ lies on the stick of
the lollipop, then $y_*$ also lies on the stick of the lollipop
corresponding to $\bar u\bar w\bar u^{-1}$ by Lemma~\ref{ztobarz} (3)
and we replace $\theta_1$ by $\bar t\bar w^{r}\bar u^{-1}$ where $\bar
t\bar w\bar u^{-1}$ is the word obtained from $twu^{-1}$ as in the
case $r=1$. It follows that the resulting word $\bar\theta_1$ is as most as long as $\theta_1$ as $|\bar t\bar w\bar u^{-1}|\le
|twu^{-1}|$ and as $|\bar w|\le |w|$ by part (2) of Lemma~\ref{ztobarz} (2).

\smallskip The remaining case is that $\theta_1=sw^{r-1}u^{-1}$ where $s$ is a suffix of $w$. Thus $\theta_1=\hat w^{r-1}su^{-1}$ where $\hat w$ is the word corresponding to the loop of the lollipop based at $y'$. Let now $\hat \theta$ be the word obtained from $su^{-1}$ as in the case $r=1$ and $\hat w'$ be the path corresponding the closed path at $y_*$ in the new lollipop that goes once around the lollipop. Note it follows from the fact that $y_*$ is at least in distance $L$ from the stick of the lollipop that $|\hat w'|\le |\hat w|$. We put $\bar \theta_1$ to be the word obtained from $(\hat w')^{r-1}\hat \theta$ by free reduction, clearly it is not longer than $\theta_1$.

Let now $S_T^*$ be the graph obtained from $S_T^+$ by performing the
above mentioned changes. 
The construction clearly preserves the fact that the map from $S_T^*$ to $\Delta_*$ is $\pi_1$-surjective. Thus $S_T^*$ folds onto $\Delta_*$ therefore onto $R_A$ as
$\Delta_*$ fords onto $R_A$.  In particular the loops of $S_T^*$ represent a tuples that is equivalent to $(a_1,\ldots ,a_k,1)$ such that the first two elements represent in $G$ the elements $b_1$ and $b_2$. Note however that the sum of the lengths of the first two petals has decreased as we first increased their lengths by at most $8L$ and then decreased their lengths by at least $20\cdot \frac{99}{100}L$. This  yields a contradiction with the minimality assumption.

\medskip 
\noindent{\bf Subcase B2: } We have $z=\bar z$ (so that
$\bar z \in F(A)$). 

Recall that $k$ is the cardinality of the free basis $A$ of $F(A)$.
If the
length of $z$ is less than $1000k^2$ then there are only $n=n(k)$
possibilities $\Gamma_1',\ldots ,\Gamma_{n(k)}'$ for the $A$-graph
$\bar\Delta=Core(\Delta)$. 
By Lemma~\ref{not-so-trivial}
for each $i=1,\dots, n(k)$ the set of words readable in $\Gamma_i'$ is
exponentially negligible in $F(A)$. Since the union of finitely many
exponentially negligible sets is again exponentially negligible, it
follows that there exists an exponentially generic subset of $F(A)$
(independent on the choice of $z$) such that for every word $w$ from that
subset no portion of $w$ that is more than $1/100$ of the length
of $w$ can be read in $Core(\Delta)$. We can then get a contradiction
similarly to the proof of  Theorem~\ref{specialcase}. Namely, by
construction there exists a freely reduced word $w\in F(A)$ (labelling one of the petals
in the graph $S_T$) whose cyclically reduced form $\tilde w$ can be
read in  $Core(\Delta)$ and such that the element represented by $w$
is conjugate to $b_1$ in $G$. Then by Corollary~\ref{cor:conj-sc} the
word $\tilde w\in F(A)$ contains a significant portion of a defining
relation from $R_\ast$, contradicting the genericity assumptions on
presentation (!!).

Suppose now that the length of $z$ is at least $1000k^2$. Then, again
arguing as in  the proof of  Theorem~\ref{specialcase}, we see that
there exists a cyclically reduced word $\tilde w\in F(A)$ that can be read in
$Core(\Delta)$ such that $\tilde w$ represents an element conjugate to
$b_1$ in $G$. Then, by Corollary~\ref{cor:conj-sc}, $\tilde w$
contains a subword $\widehat w$ such that $\widehat w$ is also a
subword of one of the defining relations in (!!) constituting more
than $1/100$ of that defining relation. Thus, on one hand, the word
$\widehat w$ is generic and, on the other hand, $\widehat w$ is equal,
as a word, to a product of the form
\[
w_0 z^{\pm 1} w_1 z^{\pm 1} \dots w_{t} z^{\pm 1} w_t
\]
where the words $w_i$ are readable in $\Psi$. Since words readable in
$\Psi$ are highly non-generic in $F(A)$ (because there exists a
reduced word of length 2 in $F(A)$ that cannot be read in $\Psi$), we
have $|w_i|\le |\widehat w|/100$ for $i=0,\dots, t$.
This implies that a definite positive proportion of $\widehat w$ consists of
the copies of $z^{\pm 1}$, which again contradicts genericity of
$\widehat w$.

This concludes the proof of Theorem~\ref{main}.


\end{document}

%% file: ra.pstex_t
\begin{picture}(0,0)%
\includegraphics{ra.pstex}%
\end{picture}%
\setlength{\unitlength}{3947sp}%
\begingroup\makeatletter\ifx\SetFigFont\undefined%
\gdef\SetFigFont#1#2#3#4#5{%
  \reset@font\fontsize{#1}{#2pt}%
  \fontfamily{#3}\fontseries{#4}\fontshape{#5}%
  \selectfont}%
\fi\endgroup%
\begin{picture}(1097,1728)(4694,-1330)
\put(5776,-511){\makebox(0,0)[lb]{\smash{{\SetFigFont{12}{14.4}{\rmdefault}{\mddefault}{\updefault}{\color[rgb]{0,0,0}$a_1$}%
}}}}
\put(5026,-1261){\makebox(0,0)[lb]{\smash{{\SetFigFont{12}{14.4}{\rmdefault}{\mddefault}{\updefault}{\color[rgb]{0,0,0}$a_2$}%
}}}}
\put(5326,239){\makebox(0,0)[lb]{\smash{{\SetFigFont{12}{14.4}{\rmdefault}{\mddefault}{\updefault}{\color[rgb]{0,0,0}$a_k$}%
}}}}
\end{picture}%

%% file: foldinginitial.pstex_t
\begin{picture}(0,0)%
\includegraphics{foldinginitial.pstex}%
\end{picture}%
\setlength{\unitlength}{3947sp}%
\begingroup\makeatletter\ifx\SetFigFont\undefined%
\gdef\SetFigFont#1#2#3#4#5{%
  \reset@font\fontsize{#1}{#2pt}%
  \fontfamily{#3}\fontseries{#4}\fontshape{#5}%
  \selectfont}%
\fi\endgroup%
\begin{picture}(5522,2556)(269,-1761)
\put(3826,-436){\makebox(0,0)[lb]{\smash{{\SetFigFont{12}{14.4}{\rmdefault}{\mddefault}{\updefault}{\color[rgb]{0,0,0}$p$}%
}}}}
\put(5776,-511){\makebox(0,0)[lb]{\smash{{\SetFigFont{12}{14.4}{\rmdefault}{\mddefault}{\updefault}{\color[rgb]{0,0,0}$a_1$}%
}}}}
\put(5026,-1261){\makebox(0,0)[lb]{\smash{{\SetFigFont{12}{14.4}{\rmdefault}{\mddefault}{\updefault}{\color[rgb]{0,0,0}$a_2$}%
}}}}
\put(5326,239){\makebox(0,0)[lb]{\smash{{\SetFigFont{12}{14.4}{\rmdefault}{\mddefault}{\updefault}{\color[rgb]{0,0,0}$a_k$}%
}}}}
\put(2401,-211){\makebox(0,0)[lb]{\smash{{\SetFigFont{12}{14.4}{\rmdefault}{\mddefault}{\updefault}{\color[rgb]{0,0,0}$g_1$}%
}}}}
\put(1876,-1336){\makebox(0,0)[lb]{\smash{{\SetFigFont{12}{14.4}{\rmdefault}{\mddefault}{\updefault}{\color[rgb]{0,0,0}$g_2$}%
}}}}
\put(676,314){\makebox(0,0)[lb]{\smash{{\SetFigFont{12}{14.4}{\rmdefault}{\mddefault}{\updefault}{\color[rgb]{0,0,0}$g_l$}%
}}}}
\end{picture}%

%% file: def-fold.pstex_t
\begin{picture}(0,0)%
\includegraphics{def-fold.pstex}%
\end{picture}%
\setlength{\unitlength}{3947sp}%
\begingroup\makeatletter\ifx\SetFigFont\undefined%
\gdef\SetFigFont#1#2#3#4#5{%
  \reset@font\fontsize{#1}{#2pt}%
  \fontfamily{#3}\fontseries{#4}\fontshape{#5}%
  \selectfont}%
\fi\endgroup%
\begin{picture}(5199,1128)(664,-3655)
\put(3301,-3061){\makebox(0,0)[lb]{\smash{{\SetFigFont{12}{14.4}{\rmdefault}{\mddefault}{\updefault}{\color[rgb]{0,0,0}$p$}%
}}}}
\put(1576,-2686){\makebox(0,0)[lb]{\smash{{\SetFigFont{12}{14.4}{\rmdefault}{\mddefault}{\updefault}{\color[rgb]{0,0,0}$a_i$}%
}}}}
\put(1576,-3586){\makebox(0,0)[lb]{\smash{{\SetFigFont{12}{14.4}{\rmdefault}{\mddefault}{\updefault}{\color[rgb]{0,0,0}$a_i$}%
}}}}
\put(4876,-3061){\makebox(0,0)[lb]{\smash{{\SetFigFont{12}{14.4}{\rmdefault}{\mddefault}{\updefault}{\color[rgb]{0,0,0}$a_i$}%
}}}}
\end{picture}%

%% file: corefoldsontora.pstex_t
\begin{picture}(0,0)%
\includegraphics{corefoldsontora.pstex}%
\end{picture}%
\setlength{\unitlength}{3947sp}%
\begingroup\makeatletter\ifx\SetFigFont\undefined%
\gdef\SetFigFont#1#2#3#4#5{%
  \reset@font\fontsize{#1}{#2pt}%
  \fontfamily{#3}\fontseries{#4}\fontshape{#5}%
  \selectfont}%
\fi\endgroup%
\begin{picture}(3261,1353)(886,-1105)
\put(2326,-586){\makebox(0,0)[lb]{\smash{{\SetFigFont{12}{14.4}{\rmdefault}{\mddefault}{\updefault}{\color[rgb]{0,0,0}$a_1$}%
}}}}
\put(901, 89){\makebox(0,0)[lb]{\smash{{\SetFigFont{12}{14.4}{\rmdefault}{\mddefault}{\updefault}{\color[rgb]{0,0,0}$a_1$}%
}}}}
\put(2326, 89){\makebox(0,0)[lb]{\smash{{\SetFigFont{12}{14.4}{\rmdefault}{\mddefault}{\updefault}{\color[rgb]{0,0,0}$a_2$}%
}}}}
\put(4051, 89){\makebox(0,0)[lb]{\smash{{\SetFigFont{12}{14.4}{\rmdefault}{\mddefault}{\updefault}{\color[rgb]{0,0,0}$a_3$}%
}}}}
\put(3826,-886){\makebox(0,0)[lb]{\smash{{\SetFigFont{12}{14.4}{\rmdefault}{\mddefault}{\updefault}{\color[rgb]{0,0,0}$a_4$}%
}}}}
\put(1726,-1036){\makebox(0,0)[lb]{\smash{{\SetFigFont{12}{14.4}{\rmdefault}{\mddefault}{\updefault}{\color[rgb]{0,0,0}$a_5$}%
}}}}
\put(1876,-511){\makebox(0,0)[lb]{\smash{{\SetFigFont{12}{14.4}{\rmdefault}{\mddefault}{\updefault}{\color[rgb]{0,0,0}$x$}%
}}}}
\put(3001,-511){\makebox(0,0)[lb]{\smash{{\SetFigFont{12}{14.4}{\rmdefault}{\mddefault}{\updefault}{\color[rgb]{0,0,0}$y$}%
}}}}
\end{picture}%

%% file: corefoldsontora+.pstex_t
\begin{picture}(0,0)%
\includegraphics{corefoldsontora+.pstex}%
\end{picture}%
\setlength{\unitlength}{3947sp}%
\begingroup\makeatletter\ifx\SetFigFont\undefined%
\gdef\SetFigFont#1#2#3#4#5{%
  \reset@font\fontsize{#1}{#2pt}%
  \fontfamily{#3}\fontseries{#4}\fontshape{#5}%
  \selectfont}%
\fi\endgroup%
\begin{picture}(4980,1836)(811,-1330)
\put(5326,239){\makebox(0,0)[lb]{\smash{{\SetFigFont{12}{14.4}{\rmdefault}{\mddefault}{\updefault}{\color[rgb]{0,0,0}$a_1$}%
}}}}
\put(5776,-511){\makebox(0,0)[lb]{\smash{{\SetFigFont{12}{14.4}{\rmdefault}{\mddefault}{\updefault}{\color[rgb]{0,0,0}$a_2$}%
}}}}
\put(5026,-1261){\makebox(0,0)[lb]{\smash{{\SetFigFont{12}{14.4}{\rmdefault}{\mddefault}{\updefault}{\color[rgb]{0,0,0}$a_3$}%
}}}}
\put(4276,-361){\makebox(0,0)[lb]{\smash{{\SetFigFont{12}{14.4}{\rmdefault}{\mddefault}{\updefault}{\color[rgb]{0,0,0}$a_1$}%
}}}}
\put(826,-211){\makebox(0,0)[lb]{\smash{{\SetFigFont{12}{14.4}{\rmdefault}{\mddefault}{\updefault}{\color[rgb]{0,0,0}$a_1$}%
}}}}
\put(901,-811){\makebox(0,0)[lb]{\smash{{\SetFigFont{12}{14.4}{\rmdefault}{\mddefault}{\updefault}{\color[rgb]{0,0,0}$a_1$}%
}}}}
\put(1651,-436){\makebox(0,0)[lb]{\smash{{\SetFigFont{12}{14.4}{\rmdefault}{\mddefault}{\updefault}{\color[rgb]{0,0,0}$a_1$}%
}}}}
\put(2176,314){\makebox(0,0)[lb]{\smash{{\SetFigFont{12}{14.4}{\rmdefault}{\mddefault}{\updefault}{\color[rgb]{0,0,0}$a_2$}%
}}}}
\put(2326,-886){\makebox(0,0)[lb]{\smash{{\SetFigFont{12}{14.4}{\rmdefault}{\mddefault}{\updefault}{\color[rgb]{0,0,0}$a_3$}%
}}}}
\end{picture}%

%% file: case1a.pstex_t
\begin{picture}(0,0)%
\includegraphics{case1a.pstex}%
\end{picture}%
\setlength{\unitlength}{3947sp}%
\begingroup\makeatletter\ifx\SetFigFontNFSS\undefined%
\gdef\SetFigFontNFSS#1#2#3#4#5{%
  \reset@font\fontsize{#1}{#2pt}%
  \fontfamily{#3}\fontseries{#4}\fontshape{#5}%
  \selectfont}%
\fi\endgroup%
\begin{picture}(3261,2471)(886,-1114)
\put(2326,-586){\makebox(0,0)[lb]{\smash{{\SetFigFontNFSS{12}{14.4}{\rmdefault}{\mddefault}{\updefault}{\color[rgb]{0,0,0}$a_1$}%
}}}}
\put(901, 89){\makebox(0,0)[lb]{\smash{{\SetFigFontNFSS{12}{14.4}{\rmdefault}{\mddefault}{\updefault}{\color[rgb]{0,0,0}$a_1$}%
}}}}
\put(2326, 89){\makebox(0,0)[lb]{\smash{{\SetFigFontNFSS{12}{14.4}{\rmdefault}{\mddefault}{\updefault}{\color[rgb]{0,0,0}$a_2$}%
}}}}
\put(4051, 89){\makebox(0,0)[lb]{\smash{{\SetFigFontNFSS{12}{14.4}{\rmdefault}{\mddefault}{\updefault}{\color[rgb]{0,0,0}$a_3$}%
}}}}
\put(3826,-886){\makebox(0,0)[lb]{\smash{{\SetFigFontNFSS{12}{14.4}{\rmdefault}{\mddefault}{\updefault}{\color[rgb]{0,0,0}$a_4$}%
}}}}
\put(1726,-1036){\makebox(0,0)[lb]{\smash{{\SetFigFontNFSS{12}{14.4}{\rmdefault}{\mddefault}{\updefault}{\color[rgb]{0,0,0}$a_5$}%
}}}}
\put(1876,-511){\makebox(0,0)[lb]{\smash{{\SetFigFontNFSS{12}{14.4}{\rmdefault}{\mddefault}{\updefault}{\color[rgb]{0,0,0}$x$}%
}}}}
\put(3001,-511){\makebox(0,0)[lb]{\smash{{\SetFigFontNFSS{12}{14.4}{\rmdefault}{\mddefault}{\updefault}{\color[rgb]{0,0,0}$y$}%
}}}}
\put(2326,1174){\makebox(0,0)[lb]{\smash{{\SetFigFontNFSS{12}{14.4}{\rmdefault}{\mddefault}{\updefault}{\color[rgb]{0,0,0}$a_1$}%
}}}}
\put(3541,1124){\makebox(0,0)[lb]{\smash{{\SetFigFontNFSS{12}{14.4}{\rmdefault}{\mddefault}{\updefault}{\color[rgb]{0,0,0}$a_1$}%
}}}}
\put(3116,104){\makebox(0,0)[lb]{\smash{{\SetFigFontNFSS{12}{14.4}{\rmdefault}{\mddefault}{\updefault}{\color[rgb]{0,0,0}$a_5$}%
}}}}
\end{picture}%

%% file: cases1b1.pstex_t
\begin{picture}(0,0)%
\includegraphics{cases1b1.pstex}%
\end{picture}%
\setlength{\unitlength}{3947sp}%
\begingroup\makeatletter\ifx\SetFigFontNFSS\undefined%
\gdef\SetFigFontNFSS#1#2#3#4#5{%
  \reset@font\fontsize{#1}{#2pt}%
  \fontfamily{#3}\fontseries{#4}\fontshape{#5}%
  \selectfont}%
\fi\endgroup%
\begin{picture}(2580,936)(61,-289)
\put(826,-136){\makebox(0,0)[lb]{\smash{{\SetFigFontNFSS{12}{14.4}{\rmdefault}{\mddefault}{\updefault}{\color[rgb]{0,0,0}$x$}%
}}}}
\put(1876, 89){\makebox(0,0)[lb]{\smash{{\SetFigFontNFSS{12}{14.4}{\rmdefault}{\mddefault}{\updefault}{\color[rgb]{0,0,0}$y$}%
}}}}
\put( 76,314){\makebox(0,0)[lb]{\smash{{\SetFigFontNFSS{12}{14.4}{\rmdefault}{\mddefault}{\updefault}{\color[rgb]{0,0,0}$a_1$}%
}}}}
\put(2626,389){\makebox(0,0)[lb]{\smash{{\SetFigFontNFSS{12}{14.4}{\rmdefault}{\mddefault}{\updefault}{\color[rgb]{0,0,0}$a_2$}%
}}}}
\put(2401,-136){\makebox(0,0)[lb]{\smash{{\SetFigFontNFSS{12}{14.4}{\rmdefault}{\mddefault}{\updefault}{\color[rgb]{0,0,0}$a_3$}%
}}}}
\put(1201,-211){\makebox(0,0)[lb]{\smash{{\SetFigFontNFSS{12}{14.4}{\rmdefault}{\mddefault}{\updefault}{\color[rgb]{0,0,0}$a_1$}%
}}}}
\put(901,389){\makebox(0,0)[lb]{\smash{{\SetFigFontNFSS{12}{14.4}{\rmdefault}{\mddefault}{\updefault}{\color[rgb]{0,0,0}$a_1$}%
}}}}
\put(1651,464){\makebox(0,0)[lb]{\smash{{\SetFigFontNFSS{12}{14.4}{\rmdefault}{\mddefault}{\updefault}{\color[rgb]{0,0,0}$a_1$}%
}}}}
\end{picture}%

%% file: cases1b2.pstex_t
\begin{picture}(0,0)%
\includegraphics{cases1b2.pstex}%
\end{picture}%
\setlength{\unitlength}{3947sp}%
\begingroup\makeatletter\ifx\SetFigFontNFSS\undefined%
\gdef\SetFigFontNFSS#1#2#3#4#5{%
  \reset@font\fontsize{#1}{#2pt}%
  \fontfamily{#3}\fontseries{#4}\fontshape{#5}%
  \selectfont}%
\fi\endgroup%
\begin{picture}(5511,1064)(-7,-410)
\put(1876, 89){\makebox(0,0)[lb]{\smash{{\SetFigFontNFSS{12}{14.4}{\rmdefault}{\mddefault}{\updefault}{\color[rgb]{0,0,0}$y$}%
}}}}
\put(1201,-211){\makebox(0,0)[lb]{\smash{{\SetFigFontNFSS{12}{14.4}{\rmdefault}{\mddefault}{\updefault}{\color[rgb]{0,0,0}$a_1$}%
}}}}
\put(1651,464){\makebox(0,0)[lb]{\smash{{\SetFigFontNFSS{12}{14.4}{\rmdefault}{\mddefault}{\updefault}{\color[rgb]{0,0,0}$a_1$}%
}}}}
\put(976, 14){\makebox(0,0)[lb]{\smash{{\SetFigFontNFSS{12}{14.4}{\rmdefault}{\mddefault}{\updefault}{\color[rgb]{0,0,0}$x$}%
}}}}
\put(2528,382){\makebox(0,0)[lb]{\smash{{\SetFigFontNFSS{12}{14.4}{\rmdefault}{\mddefault}{\updefault}{\color[rgb]{0,0,0}$a_2$}%
}}}}
\put(211,254){\makebox(0,0)[lb]{\smash{{\SetFigFontNFSS{12}{14.4}{\rmdefault}{\mddefault}{\updefault}{\color[rgb]{0,0,0}$a_1$}%
}}}}
\put(338,-188){\makebox(0,0)[lb]{\smash{{\SetFigFontNFSS{12}{14.4}{\rmdefault}{\mddefault}{\updefault}{\color[rgb]{0,0,0}$a_3$}%
}}}}
\put(901,389){\makebox(0,0)[lb]{\smash{{\SetFigFontNFSS{12}{14.4}{\rmdefault}{\mddefault}{\updefault}{\color[rgb]{0,0,0}$a_2^\eta$}%
}}}}
\put(3008,-263){\makebox(0,0)[lb]{\smash{{\SetFigFontNFSS{12}{14.4}{\rmdefault}{\mddefault}{\updefault}{\color[rgb]{0,0,0}$a_3$}%
}}}}
\put(3061,471){\makebox(0,0)[lb]{\smash{{\SetFigFontNFSS{12}{14.4}{\rmdefault}{\mddefault}{\updefault}{\color[rgb]{0,0,0}$a_1$}%
}}}}
\put(4081,-302){\makebox(0,0)[lb]{\smash{{\SetFigFontNFSS{12}{14.4}{\rmdefault}{\mddefault}{\updefault}{\color[rgb]{0,0,0}$a_1$}%
}}}}
\put(4111, 96){\makebox(0,0)[lb]{\smash{{\SetFigFontNFSS{12}{14.4}{\rmdefault}{\mddefault}{\updefault}{\color[rgb]{0,0,0}$a_2^\eta$}%
}}}}
\put(4568,321){\makebox(0,0)[lb]{\smash{{\SetFigFontNFSS{12}{14.4}{\rmdefault}{\mddefault}{\updefault}{\color[rgb]{0,0,0}$a_2^\eta/a_1$}%
}}}}
\put(3864,411){\makebox(0,0)[lb]{\smash{{\SetFigFontNFSS{12}{14.4}{\rmdefault}{\mddefault}{\updefault}{\color[rgb]{0,0,0}$a_2^\eta$}%
}}}}
\end{picture}%

%% file: delta.pstex_t
\begin{picture}(0,0)%
\includegraphics{delta.pstex}%
\end{picture}%
\setlength{\unitlength}{3947sp}%
\begingroup\makeatletter\ifx\SetFigFont\undefined%
\gdef\SetFigFont#1#2#3#4#5{%
  \reset@font\fontsize{#1}{#2pt}%
  \fontfamily{#3}\fontseries{#4}\fontshape{#5}%
  \selectfont}%
\fi\endgroup%
\begin{picture}(5836,2205)(494,-1983)
\end{picture}%

%% file: delta2.pstex_t
\begin{picture}(0,0)%
\includegraphics{delta2.pstex}%
\end{picture}%
\setlength{\unitlength}{3947sp}%
\begingroup\makeatletter\ifx\SetFigFont\undefined%
\gdef\SetFigFont#1#2#3#4#5{%
  \reset@font\fontsize{#1}{#2pt}%
  \fontfamily{#3}\fontseries{#4}\fontshape{#5}%
  \selectfont}%
\fi\endgroup%
\begin{picture}(5851,2205)(494,-1983)
\put(1651,-211){\makebox(0,0)[lb]{\smash{{\SetFigFont{12}{14.4}{\rmdefault}{\mddefault}{\updefault}{\color[rgb]{0,0,0}$y=y'$}%
}}}}
\put(5251,-211){\makebox(0,0)[lb]{\smash{{\SetFigFont{12}{14.4}{\rmdefault}{\mddefault}{\updefault}{\color[rgb]{0,0,0}$y$}%
}}}}
\put(4951,-661){\makebox(0,0)[lb]{\smash{{\SetFigFont{12}{14.4}{\rmdefault}{\mddefault}{\updefault}{\color[rgb]{0,0,0}$y'$}%
}}}}
\end{picture}%

%% file: deltaunfold.pstex_t
\begin{picture}(0,0)%
\includegraphics{deltaunfold.pstex}%
\end{picture}%
\setlength{\unitlength}{3947sp}%
\begingroup\makeatletter\ifx\SetFigFont\undefined%
\gdef\SetFigFont#1#2#3#4#5{%
  \reset@font\fontsize{#1}{#2pt}%
  \fontfamily{#3}\fontseries{#4}\fontshape{#5}%
  \selectfont}%
\fi\endgroup%
\begin{picture}(5551,1821)(794,-1885)
\put(1801,-436){\makebox(0,0)[lb]{\smash{{\SetFigFont{12}{14.4}{\rmdefault}{\mddefault}{\updefault}{\color[rgb]{0,0,0}$y$}%
}}}}
\put(5326,-436){\makebox(0,0)[lb]{\smash{{\SetFigFont{12}{14.4}{\rmdefault}{\mddefault}{\updefault}{\color[rgb]{0,0,0}$y$}%
}}}}
\put(901,-1036){\makebox(0,0)[lb]{\smash{{\SetFigFont{12}{14.4}{\rmdefault}{\mddefault}{\updefault}{\color[rgb]{0,0,0}$v_0$}%
}}}}
\put(4426,-1036){\makebox(0,0)[lb]{\smash{{\SetFigFont{12}{14.4}{\rmdefault}{\mddefault}{\updefault}{\color[rgb]{0,0,0}$v_0$}%
}}}}
\put(4951,-811){\makebox(0,0)[lb]{\smash{{\SetFigFont{12}{14.4}{\rmdefault}{\mddefault}{\updefault}{\color[rgb]{0,0,0}$y'$}%
}}}}
\put(1426,-811){\makebox(0,0)[lb]{\smash{{\SetFigFont{12}{14.4}{\rmdefault}{\mddefault}{\updefault}{\color[rgb]{0,0,0}$y'$}%
}}}}
\put(1051,-211){\makebox(0,0)[lb]{\smash{{\SetFigFont{12}{14.4}{\rmdefault}{\mddefault}{\updefault}{\color[rgb]{0,0,0}$u$}%
}}}}
\put(4501,-361){\makebox(0,0)[lb]{\smash{{\SetFigFont{12}{14.4}{\rmdefault}{\mddefault}{\updefault}{\color[rgb]{0,0,0}$u$}%
}}}}
\put(1501,-443){\makebox(0,0)[lb]{\smash{{\SetFigFont{12}{14.4}{\rmdefault}{\mddefault}{\updefault}{\color[rgb]{0,0,0}$v$}%
}}}}
\end{picture}%

%% file: loopofs.pstex_t
\begin{picture}(0,0)%
\includegraphics{loopofs.pstex}%
\end{picture}%
\setlength{\unitlength}{3947sp}%
\begingroup\makeatletter\ifx\SetFigFont\undefined%
\gdef\SetFigFont#1#2#3#4#5{%
  \reset@font\fontsize{#1}{#2pt}%
  \fontfamily{#3}\fontseries{#4}\fontshape{#5}%
  \selectfont}%
\fi\endgroup%
\begin{picture}(2400,1134)(-494,-750)
\put(-82, 67){\makebox(0,0)[lb]{\smash{{\SetFigFont{12}{14.4}{\rmdefault}{\mddefault}{\updefault}{\color[rgb]{0,0,0}$u$}%
}}}}
\put(1891,-113){\makebox(0,0)[lb]{\smash{{\SetFigFont{12}{14.4}{\rmdefault}{\mddefault}{\updefault}{\color[rgb]{0,0,0}$d_i$}%
}}}}
\put(-142,-525){\makebox(0,0)[lb]{\smash{{\SetFigFont{12}{14.4}{\rmdefault}{\mddefault}{\updefault}{\color[rgb]{0,0,0}$u$}%
}}}}
\put(-479,-277){\makebox(0,0)[lb]{\smash{{\SetFigFont{12}{14.4}{\rmdefault}{\mddefault}{\updefault}{\color[rgb]{0,0,0}$v_0$}%
}}}}
\end{picture}%

%% file: gammaunfold.pstex_t
\begin{picture}(0,0)%
\includegraphics{gammaunfold.pstex}%
\end{picture}%
\setlength{\unitlength}{3947sp}%
\begingroup\makeatletter\ifx\SetFigFont\undefined%
\gdef\SetFigFont#1#2#3#4#5{%
  \reset@font\fontsize{#1}{#2pt}%
  \fontfamily{#3}\fontseries{#4}\fontshape{#5}%
  \selectfont}%
\fi\endgroup%
\begin{picture}(5587,1775)(-374,-965)
\put(1621,-249){\makebox(0,0)[lb]{\smash{{\SetFigFont{12}{14.4}{\rmdefault}{\mddefault}{\updefault}{\color[rgb]{0,0,0}$d_i$}%
}}}}
\put(3399,329){\makebox(0,0)[lb]{\smash{{\SetFigFont{12}{14.4}{\rmdefault}{\mddefault}{\updefault}{\color[rgb]{0,0,0}$v$}%
}}}}
\put(3046,-45){\makebox(0,0)[lb]{\smash{{\SetFigFont{12}{14.4}{\rmdefault}{\mddefault}{\updefault}{\color[rgb]{0,0,0}$u$}%
}}}}
\put(3097,-496){\makebox(0,0)[lb]{\smash{{\SetFigFont{12}{14.4}{\rmdefault}{\mddefault}{\updefault}{\color[rgb]{0,0,0}$u$}%
}}}}
\put(3661,-720){\makebox(0,0)[lb]{\smash{{\SetFigFont{12}{14.4}{\rmdefault}{\mddefault}{\updefault}{\color[rgb]{0,0,0}$v$}%
}}}}
\put(3286,-232){\makebox(0,0)[lb]{\smash{{\SetFigFont{12}{14.4}{\rmdefault}{\mddefault}{\updefault}{\color[rgb]{0,0,0}$v_0$}%
}}}}
\put(3961,532){\makebox(0,0)[lb]{\smash{{\SetFigFont{12}{14.4}{\rmdefault}{\mddefault}{\updefault}{\color[rgb]{0,0,0}$p_i$}%
}}}}
\put(4223,-788){\makebox(0,0)[lb]{\smash{{\SetFigFont{12}{14.4}{\rmdefault}{\mddefault}{\updefault}{\color[rgb]{0,0,0}$q_i$}%
}}}}
\put(5198, 37){\makebox(0,0)[lb]{\smash{{\SetFigFont{12}{14.4}{\rmdefault}{\mddefault}{\updefault}{\color[rgb]{0,0,0}$\bar d_i$}%
}}}}
\put(-127,-292){\makebox(0,0)[lb]{\smash{{\SetFigFont{12}{14.4}{\rmdefault}{\mddefault}{\updefault}{\color[rgb]{0,0,0}$v_0$}%
}}}}
\put(-359,-67){\makebox(0,0)[lb]{\smash{{\SetFigFont{12}{14.4}{\rmdefault}{\mddefault}{\updefault}{\color[rgb]{0,0,0}$u$}%
}}}}
\put(-338,-578){\makebox(0,0)[lb]{\smash{{\SetFigFont{12}{14.4}{\rmdefault}{\mddefault}{\updefault}{\color[rgb]{0,0,0}$u$}%
}}}}
\put(-126,329){\makebox(0,0)[lb]{\smash{{\SetFigFont{12}{14.4}{\rmdefault}{\mddefault}{\updefault}{\color[rgb]{0,0,0}$v$}%
}}}}
\put(369,509){\makebox(0,0)[lb]{\smash{{\SetFigFont{12}{14.4}{\rmdefault}{\mddefault}{\updefault}{\color[rgb]{0,0,0}$v$}%
}}}}
\put(474,-901){\makebox(0,0)[lb]{\smash{{\SetFigFont{12}{14.4}{\rmdefault}{\mddefault}{\updefault}{\color[rgb]{0,0,0}$v$}%
}}}}
\put( 24,-848){\makebox(0,0)[lb]{\smash{{\SetFigFont{12}{14.4}{\rmdefault}{\mddefault}{\updefault}{\color[rgb]{0,0,0}$v$}%
}}}}
\end{picture}%

%% file: gammastern.pstex_t
\begin{picture}(0,0)%
\includegraphics{gammastern.pstex}%
\end{picture}%
\setlength{\unitlength}{3947sp}%
\begingroup\makeatletter\ifx\SetFigFont\undefined%
\gdef\SetFigFont#1#2#3#4#5{%
  \reset@font\fontsize{#1}{#2pt}%
  \fontfamily{#3}\fontseries{#4}\fontshape{#5}%
  \selectfont}%
\fi\endgroup%
\begin{picture}(2703,3073)(-119,-1983)
\put(-104,-1208){\makebox(0,0)[lb]{\smash{{\SetFigFont{12}{14.4}{\rmdefault}{\mddefault}{\updefault}{\color[rgb]{0,0,0}$v_0$}%
}}}}
\put(1321,-68){\makebox(0,0)[lb]{\smash{{\SetFigFont{12}{14.4}{\rmdefault}{\mddefault}{\updefault}{\color[rgb]{0,0,0}$y$}%
}}}}
\put(1943,-698){\makebox(0,0)[lb]{\smash{{\SetFigFont{12}{14.4}{\rmdefault}{\mddefault}{\updefault}{\color[rgb]{0,0,0}$y_*$}%
}}}}
\end{picture}%